\documentclass[12pt]{amsart}

\usepackage{amsmath,verbatim,amsfonts,mathtools,enumitem,mnsymbol}

\DeclareMathOperator*{\topmodule}{mod}

\newcommand{\Adele}{\ensuremath{\mathbb{A}}}

\newcommand{\nonarchimedean}{non-Archimedean} 
\newcommand{\isomorphic}{\ensuremath{\cong}}

\newcommand{\mathe}{\ensuremath{\mathrm{e}}}
\newcommand{\rconj}[1]{\ensuremath{{}^{#1}}}
\newcommand{\transpose}[1]{\ensuremath{\rconj{t}#1}}

\def\Alg{\operatorname{Alg}}
\def\Ind{\operatorname{Ind}}
\def\ind{\operatorname{ind}}
\def\bsl{\backslash}
\def\NN{{\mathbb N}}

\def\QQ{{\mathbb Q}}

\def\CC{{\mathbb C}}
\def\GL{\mathrm{GL}}

\def\Hom{\mathrm{Hom}}
\def\Tri{\mathrm{Tri}}

\def\today{\ifcase\month\or
 January\or February\or March\or April\or May\or June\or
 July\or August\or September\or October\or November\or December\fi
 \space\number\day, \number\year}

\makeatletter
\@addtoreset{equation}{section}

\newtheorem{theorem}{Theorem}
\newtheorem{conjecture}{Conjecture}

\newtheorem{lemma}{Lemma}[section]
\newtheorem{proposition}[lemma]{Proposition}
\newtheorem{corollary}[lemma]{Corollary}
\newtheorem{claim}[lemma]{Claim}
\newtheorem{remark}{Remark}[section]
\newtheorem{example}{Example}[section]
\numberwithin{equation}{section}

\begin{document}

\title[Distinguished representations]{Representations distinguished by pairs of exceptional representations and a conjecture of Savin}

\author{Eyal Kaplan}
\email{kaplaney@gmail.com}

\keywords{Savin's conjecture, distinguished representations, exceptional representations}
\thanks{The author is partially supported by the ISF Center of Excellence grant \#1691/10.}

\begin{abstract}
We study representations of $\GL_n$ appearing as quotients of a tensor of exceptional representations, in the
sense of Kazhdan and Patterson. Such representations are called distinguished. We characterize distinguished principal series representations in terms of their inducing data. In particular, we complete the proof of a conjecture of Savin, relating distinguished spherical
representations to the image of the tautological lift from a suitable classical group.
\end{abstract}

\maketitle

\section{Introduction}\label{section:introduction}
Let $F$ be a local \nonarchimedean\ field of characteristic different from $2$. Let $\tau$ be an admissible representation
of $\GL_n(F)$ and let $\theta$ and $\theta'$ be a pair of exceptional representations in the sense of Kazhdan and Patterson \cite{KP}, of the metaplectic double cover $\widetilde{\GL}_n(F)$
of $\GL_n(F)$. The representation $\tau$ is called distinguished if there is a nonzero trilinear form on the space of $\tau\times\theta\times\theta'$, which is $\GL_n(F)$-invariant. Equivalently,
\begin{align}\label{space:hom space}
\Hom_{\GL_n(F)}(\theta\otimes\theta',\tau^{\vee})\ne0.
\end{align}
Here $\tau^{\vee}$ is the representation contragradient to $\tau$.

We study distinguished representations. The main result of this work is the following combinatorial characterization, of
irreducible distinguished principal series representations. We say that a character $\eta=\eta_1\otimes\ldots\otimes\eta_n$ of the diagonal torus satisfies
condition $(\star)$ if, up to a permutation of the characters $\eta_i$, there is $0\leq k\leq\lfloor n/2\rfloor$ such that
\begin{itemize}
\item $\eta_{2i}=\eta_{2i-1}^{-1}$ for $1\leq i\leq k$,
\item $\eta_i^2=1$ for $2k+1\leq i\leq n$.
\end{itemize}

\begin{theorem}\label{theorem:distinguished principal series}
Let $\tau$ be a principal series representation of $\GL_n(F)$, induced from the character
$\eta$. If $\tau$ is distinguished, $\eta$ satisfies $(\star)$. Conversely, if $(\star)$ holds and $\tau$ is irreducible,
then $\tau$ is distinguished.
\end{theorem}

The main corollary of this theorem, is the validity of the ``only if" part of the following conjecture of Savin \cite{Savin3}.
\begin{theorem}\label{theorem:Savin's conjecture}
Let $\tau$ be a spherical representation of $\GL_n(F)$, i.e., an irreducible unramified quotient of some principal series representation, with a trivial central character. Then $\tau$ is distinguished if and only if $\tau$ is the lift of a representation
of (split) $SO_{2\lfloor n/2\rfloor}(F)$ if $n$ is even or $Sp_{2\lfloor n/2\rfloor}(F)$ if $n$ is odd.
\end{theorem}

Both theorems currently hold under the plausible assumption that for $n>3$, the exceptional representations do not have Whittaker models. This was proved by Kazhdan and Patterson for local fields of odd residual characteristic (\cite{KP} Section~I.3, see also \cite{Kable} Theorem~5.4). The remaining case - a field of characteristic $0$ and even residual characteristic, is expected to be completed through progress in the trace formula (\cite{BG} p.~138, see \cite{FKS} Lemma~6).

For $n=3$, Theorem~\ref{theorem:Savin's conjecture} was proved (unconditionally) by Savin \cite{Savin3},
who analyzed the dimension of \eqref{space:hom space} for an arbitrary irreducible quotient of a principal series representation of $\GL_3(F)$.

The ``if" part of Theorem~\ref{theorem:Savin's conjecture} (for any $n$) was proved by Kable \cite{Kable2}. We briefly describe his approach. Kable used analytic techniques, resembling the methods of \cite{CS2,BFF}. He started with a ``pseudo integral" $\Upsilon$ on $\tau\times\theta\times\theta'$ satisfying the equivariance properties, defined for inducing data $\eta$ in an appropriate cone of convergence. As a function of $\eta$, $\Upsilon$ was a polynomial in the Satake parameters of $\eta$ and $\eta^{-1}$. Using Bernstein's continuation principle (\cite{Banks}) $\Upsilon$ was extended to all characters. The key for using this principle is a one-dimensionality result, namely, the space \eqref{space:hom space} is at most one dimensional, for a ``large enough"
subset of characters. This was proved by Kable \cite{Kable} (Theorem~6.4) using the theory of derivatives of Bernstein and Zelevinsky \cite{BZ1,BZ2}. The assumption concerning the absence of Whittaker models for $n>3$ was needed also for his results.
We also mention that Kable \cite{Kable} (Theorem~6.3) already proved that if a principal series representation is distinguished,
the character $\eta$ satisfies a combinatorial condition, different from $(\star)$ in that $\eta_{2i}$ and $\eta_{2i-1}^{-1}$ only agree on $(F^*)^2$, that is, $\eta_{2i}^2=\eta_{2i-1}^{-2}$. The methods used for his proof seem insufficient to deduce the precise result (\cite{Kable2} p.~1602).

Condition $(\star)$ implies that distinguished representations do not enjoy complete hereditary properties, in the sense that
induction from distinguished representations does not exhaust all distinguished representations. We do have upper heredity, given by the following result.
\begin{theorem}\label{theorem:upper heredity of dist}
Let $\tau_1$ and $\tau_2$ be a pair of distinguished representations of $\GL_{n_1}(F)$ and $\GL_{n_2}(F)$. The representation $\tau$ 
parabolically induced from $\tau_1\otimes\tau_2$ is distinguished. 
\end{theorem}


Theorems~\ref{theorem:distinguished principal series} and \ref{theorem:upper heredity of dist} and similar results (e.g. \cite{Matringe,FLO}, see below) motivate the following conjecture, giving the combinatorial characterization of irreducible generic distinguished representations.
\begin{conjecture}\label{conjecture:combinatorial classification}
Let $\tau$ be an irreducible generic representation of $\GL_n(F)$. Then
$\tau$ is distinguished if and only if $\tau$ is isomorphic to a representation parabolically induced from a representation $\Delta_1\otimes\ldots\otimes\Delta_m$ of $\GL_{n_1}(F)\times\ldots\times\GL_{n_m}(F)$ with the following properties:
\begin{itemize}[leftmargin=*]
\item Each $\Delta_i$ is essentially square integrable,
\item There is $0\leq k\leq\lfloor n/2\rfloor$ such that $\Delta_{2i}=\Delta_{2i-1}^{\vee}$ for $1\leq i\leq k$,
\item The representation $\Delta_i$ is distinguished for $2k+1\leq i\leq n$.
\end{itemize}
\end{conjecture}
In the case of a principal series representation, the conjecture becomes Theorem~\ref{theorem:distinguished principal series} because for  $n=1$, an irreducible representation $\Delta$ is distinguished if and only if $\Delta^2=1$.

The exceptional representations of $\widetilde{\GL}_n(F)$ were first defined and studied by Kazhdan and Patterson \cite{KP}. Their motivation was global, to study a class of automorphic forms on the metaplectic group. 
One of the significant applications of their theory
was the construction of a Rankin-Selberg integral representation for the symmetric square $L$-function by Bump and Ginzburg \cite{BG}.

Let $\mathrm{k}$ be a number field with a ring of ad\`{e}les $\Adele$. Let $\pi$ be a cuspidal automorphic representation of
$\GL_n(\Adele)$ with a unitary central character. In their seminal work, Bump and Ginzburg \cite{BG} showed that the only possible poles of the partial $L$-function $L_S(s,\pi,\mathrm{Sym}^2)$ are at $s=0,1$ and furthermore, the existence of a
pole at $s=1$ implies the nonvanishing of a period integral of the form
\begin{align}\label{int:BG period integral}
\int_{{Z'\GL_{n}(F)}\bsl{\GL_{n}(\Adele)}}\rho(m)\varphi(m)\varphi'(m)dm.
\end{align}
Here $Z'$ is a subgroup of finite index in the center $Z_n(\Adele)$ of $\GL_n(\Adele)$, $\rho$ is a cusp form in the space of $\pi$ and
$\varphi$ and $\varphi'$ are automorphic forms corresponding to the global exceptional representations $\theta$ and $\theta'$. Takeda \cite{Tk} extended the results of \cite{BG} to the twisted symmetric square $L$-function, but did not consider a period integral.

Now consider an irreducible supercuspidal representation $\tau$ such that the (local) $L$-function $L(s,\tau,\mathrm{Sym}^2)$ has a pole at $s=0$. As an application of the descent method of Ginzburg, Rallis and Soudry 
(\cite{GRS2,GRS5,GRS7,GRS3,GRS8,JSd1,JSd2,Soudry5,Soudry4,RGS}), one can globalize $\tau$ to a cuspidal automorphic representation
$\pi$ of $\GL_n(\Adele)$, such that $L_S(s,\pi,\mathrm{Sym}^2)$ has a pole at $s=1$ (see the appendix of \cite{PR}). Therefore \eqref{int:BG period integral} implies that $\tau$ is distinguished.

This is an example of the analytic approach to the study of distinguished representations. In an ongoing work by Shunsuke Yamana and the author,
we develop a global theory of distinguished representations and extend the results of \cite{BG,Tk}. In the case of even $n$, we present a novel integral representation for the twisted symmetric square $L$-function. We characterize the pole at $s=1$ in terms of a period integral similar to \eqref{int:BG period integral}. Furthermore, we determine the irreducible distinguished summands of the discrete spectrum of $\GL_n$.

The case of $\GL_n$ can be placed in a more general context. The following co-period integral was studied in \cite{me7},
\begin{align*}
\int_{{SO_{2n+1}(\mathrm{k})}\bsl{SO_{2n+1}(\Adele)}}\mathrm{Res}_{s=1/2}\mathrm{E}(g;\rho,s)\Phi(g)\Phi'(g)dg.
\end{align*}
Here $\mathrm{E}(g;\rho,s)$ is an Eisenstein series corresponding to an element $\rho$, in the space of the representation
of $SO_{2n+1}(\Adele)$, parabolically induced from an automorphic cuspidal representation $\pi$ of $\GL_n(\Adele)$, $\Phi$ and $\Phi'$ are automorphic forms in the space of the small representation of $\widetilde{SO}_{2n+1}(\Adele)$ of Bump, Friedberg and Ginzburg \cite{BFG}. The nonvanishing of this integral was related to the nonvanishing of a period similar to \eqref{int:BG period integral}. This global result has a local counterpart, showing that a local representation $I(\tau,s)$ of $SO_{2n+1}(F)$, induced from a Siegel parabolic subgroup and $\tau|\det|^s$, is distinguished at $s=1/2$ whenever $\tau$ is distinguished.

Exceptional representations are related to a broader class of small, or minimal, representations. Perhaps the first example was the Weil representation of $\widetilde{Sp}_{2n}$. These representations played a fundamental role in constructions of lifts and Rankin-Selberg integrals. They are extremely useful for applications, mainly because they enjoy the vanishing of a large class of twisted Jacquet modules, or globally phrased, Fourier coefficients \cite{GRS6,G2,GJS2}. Minimal or small representations have been studied and used by many authors, including \cite{V,KZ,KS,Savin2,BK,Savin,GRS,BFG3,GRS3,KP3,BFG,JSd1,GanSavin,Soudry4,LokeSavin,RGS}.

The term ``distinguished" has been used in the following context. Let $\xi$ be a representation of a group $G$ and let $\eta$ be a character of a subgroup $H<G$. Then $\xi$ is said to be $(H,\eta)$-distinguished if $\Hom_H(\xi,\eta)\ne0$.
There are numerous studies on local and global distinguished representations, including \cite{Jac2,JR,FJ,Offen,OS2,OS,Offen2,Jac3,Matringe3,Matringe2,Matringe,FLO,Matringe5}.

Matringe \cite{Matringe3,Matringe2,Matringe} studied representations of $\GL_n(F_0)$, where $F_0$ is a quadratic extension of $F$, which are $(\GL_n(F),\eta)$-distinguished. He proved (\cite{Matringe2}) that an irreducible generic representation $\xi$ is distinguished, if and only if its Rankin-Selberg Asai $L$-function $L(s,\xi,\mathrm{Asai})$ has an exceptional pole at $0$. Matringe also proved a combinatorial classification result (\cite{Matringe} Theorem~5.2) similar to Conjecture~\ref{conjecture:combinatorial classification}, which he used in \cite{Matringe4} to prove $L(s,\xi,\mathrm{Asai})=L(s,\rho(\xi),\mathrm{Asai})$, where $\rho(\xi)$ is the Langlands parameter associated with $\xi$.

Feigon, Lapid and Offen \cite{FLO} studied representations distinguished by unitary groups, locally and globally.

One tool, used repeatedly for local analysis in the aforementioned works, is Mackey theory, or a variant on the Geometric Lemma of Bernstein and Zelevinsky \cite{BZ2}. For example, Matringe \cite{Matringe} considered the filtration of a representation parabolically induced from
$Q(F_0)$ to $\GL_n(F_0)$, as a $\GL_n(F)$-module. 

In our setting we look at the structure of the $\GL_n(F)$-module $\theta\otimes\theta'$ and exploit the properties of exceptional representations. Our analysis of the space \eqref{space:hom space} is based on ideas of Savin \cite{Savin3} and Kable \cite{Kable}. To study $\theta\otimes\theta'$, we extend the filtration argument in \cite{Savin3} to any $n$. However, Savin used a geometric model for $\theta$ (\cite{FKS}), which is not available in general.
We utilize the computation of derivatives of $\theta$ in \cite{Kable}, to reduce several problems to questions on modules of the mirabolic subgroup. We develop certain extensions to the functors of Bernstein and Zelevinsky \cite{BZ2} (Section~3), which may be of independent interest.

One delicate point about exceptional representations, is that when $n$ is even, the center $Z_n(F)$ of $\GL_n(F)$ is not central in the cover. In turn $\theta$ does not admit a character of $\widetilde{Z}_n(F)$. This will require extra care in our inductive argument (see the proof of Proposition~\ref{proposition:only if direction, for any principal series}). It is interesting to note that when $n=2$, $\widetilde{Z}_n(F)$ acts by a character on the second derivative of $\theta$. This was proved by Gelbart and Piatetski-Shapiro \cite{GP} (Theorem~2.2), and was used by Kable \cite{Kable} (Theorem~5.3) to compute the second derivative of $\theta$.



The relation between equivariant trilinear forms and $\epsilon$-factors has been studied by Prasad \cite{Prasad}.

The rest of this work is organized as follows. Section~\ref{section:preliminaries} contains preliminaries and notation.
In Section~\ref{section:geometric} we prove several technical results, concerning representations of the mirabolic subgroup, that will be used for the computations of Jacquet modules. Our main results on distinguished representations occupy Section~\ref{section:Distinguished representations}.

\section{Preliminaries}\label{section:preliminaries}

\subsection{The groups}\label{subsection:the groups}
Let $F$ be a local \nonarchimedean\ field of characteristic different from $2$. 
Let $(,)$ be the Hilbert symbol of order $2$ of $F$ and put $\mu_2=\{-1,1\}$.
We usually denote by $\psi$ a fixed nontrivial additive character of $F$ and then
$\gamma_{\psi}$ is the normalized Weil factor (\cite{We} Section~14, $\gamma_{\psi}(a)$ is
$\gamma_F(a,\psi)$ in the notation of \cite{Rao}, $\gamma_{\psi}(\cdot)^4=1$).

In the group $\GL_{n}$, fix the Borel subgroup of upper triangular invertible matrices
$B_{n}=T_{n}\ltimes N_{n}$, where $T_{n}$ is the diagonal torus. If $m_1,\ldots,m_l\geq0$ satisfy $m_1+\ldots+m_l=n$,
let $Q_{m_1,\ldots,m_l}=M_{m_1,\ldots,m_l}\ltimes U_{m_1,\ldots,m_l}$ be the standard maximal parabolic subgroup with a Levi part
$M_{m_1,\ldots,m_l}\isomorphic\GL_{m_1}\times\ldots\times\GL_{m_l}$. Let $Z_n$ be the center of $\GL_n$. Denote by $Y_n$
the mirabolic subgroup of $\GL_n$, that is, the subgroup of elements whose last row is $(0,\ldots,0,1)$. Also
denote by $I_{n}$ the identity matrix of $\GL_{n}(F)$.
For a parabolic subgroup $Q<\GL_n$, let $\delta_{Q(F)}$ denote the modulus character of $Q(F)$. For any $l\leq n$, $\GL_l$ is embedded
in $\GL_n$ in the top left corner. For $d\in \NN$ and $H<T_n(F)$, put
$H^d=\{h^d:h\in H\}$. Also put $F^{*d}=(F^*)^d$. If $x,y\in \GL_n(F)$ and $Y<\GL_n(F)$, $\rconj{x}y=xyx^{-1}$ and $\rconj{x}Y=\{\rconj{x}y:y\in Y\}$.

Let $\widetilde{\GL}_n(F)$ be the metaplectic double cover of $\GL_n(F)$, as constructed by
Kazhdan and Patterson \cite{KP} (with their $c$ parameter equal to $0$). Recall that they defined their cover using an
embedding of $\GL_n(F)$ in $SL_{n+1}(F)$, and the cover of $SL_{n+1}(F)$ of Matsumoto \cite{Mats}. We use the block-compatible cocycle of Banks, Levi and Sepanski \cite{BLS}. If $n=2$, this cocycle coincides with the one given by Kubota \cite{Kubota}. Let $p:\widetilde{\GL}_n(F)\rightarrow\GL_n(F)$ be the natural projection.
For any subset $X\subset \GL_n(F)$, denote $\widetilde{X}=p^{-1}(X)$.
Let $\mathe$ be $1$ if $n$ is odd, otherwise $\mathe=2$. Then $\widetilde{Z}_n(F)^{\mathe}$ is the center of $\widetilde{\GL}_n(F)$.

Henceforth we exclude the field $F$ from the notation.

\subsection{Representations}\label{subsection:representations}
Let $G$ be an $l$-group (\cite{BZ1} 1.1). Throughout, representations of $G$ will be complex and smooth. We let
$\Alg{G}$ denote the category of these representations.
If $\pi$ is a representation of $G$, $\pi^{\vee}$ is the representation contragradient to $\pi$.
The central character of $\pi$, if exists, is denoted $\omega_{\pi}$.

Regular induction is denoted $\Ind$ while $\ind$ is the compact induction. When inducing from a parabolic subgroup, induction is always
taken to be normalized.

We say that $\pi$ is glued from representations $\pi_1,\ldots,\pi_l$ if $\pi$ has a filtration whose quotients are, after a permutation, $\pi_1,\ldots,\pi_l$. For convenience, we also write $\pi=\mathrm{s.s.}\bigoplus_{i=1}^l\pi_i$ and refer to both sides as $G$-modules (of course each $\pi_i$ is a $G$-module, but the right-hand side is not a direct sum).
The representations $\pi_i$ might be isomorphic or zero.

If $\pi$ is a representation of a subgroup $H<G$ and $w\in G$, denote by $\rconj{w}\pi$ the representation of $\rconj{w}H$ on the
space of $\pi$ acting by $\rconj{w}\pi(h)=\pi(\rconj{w^{-1}}h)$.

If $\pi$ and $\pi'$ are a pair of genuine representations of $\widetilde{\GL}_n$, their (outer) tensor product
$\pi\otimes \pi'$ can be regarded as a representation of $\GL_n$ by $g\mapsto\pi(\varphi(g))\otimes\pi'(\varphi(g))$, where
$\varphi:\GL_n\rightarrow\widetilde{\GL}_n$ is an arbitrary section. The actual choice of $\varphi$ does not matter, hence
it will usually be omitted.

\subsection{Filtration of induced represenations}\label{subsection:filtration of induced representations}
We recall the increasing
filtration of induced representations of Bernstein and Zelevinsky \cite{BZ1} (2.24).
Let $G$ be an $l$-group, $H<G$ be a closed subgroup and $\pi$ be a representation of $H$ on a space $E$.
Denote the space of the induced representation $\ind_H^G(\pi)$ by $W$. For a compact open subgroup $\mathcal{V}<G$,
let $W^{\mathcal{V}}$ be the subspace of vectors
invariant by $\mathcal{V}$. Choose a set of representatives $\Omega_{\mathcal{V}}$ for $H\bsl G/\mathcal{V}$.
Then $W^{\mathcal{V}}$ is linearly isomorphic with the space of functions $f:\Omega_{\mathcal{V}}\rightarrow E$,
such that $f(g)$ is invariant by $H\cap\rconj{g}\mathcal{V}$ for all $g\in\Omega_{\mathcal{V}}$, and $f$ vanishes outside of
a finite subset of $\Omega_{\mathcal{V}}$. Then if
\begin{align*}
\mathcal{V}_1>\ldots>\mathcal{V}_l>\ldots
\end{align*}
is a decreasing sequence of compact open subgroups,
\begin{align*}
W^{\mathcal{V}_1}\subset\ldots\subset W^{\mathcal{V}_l}\subset\ldots
\end{align*}
is an increasing filtration of $W$. A similar result holds for $\Ind_H^G(\pi)$, except that the functions $f$ may be nonzero
on an infinite number of representatives from $\Omega_{\mathcal{V}}$.

\subsection{Jacquet modules}\label{subsection:Jacquet modules}
Let $\pi$ be a representation of an $l$-group $G$ on a space $E$. If $U$ is a unipotent subgroup, which is exhausted by its compact subgroups (always the case for $U<\GL_n$), let $E(U)\subset E$ be the subspace generated by
the vectors $\pi(u)v-v$ where $u\in U$ and $v\in E$. Put $E_U=E(U)\bsl E$. If $M$ is the normalizer of $U$ in $G$, the
following sequence of $M$-modules
\begin{align*}
0\rightarrow E(U)\rightarrow E\rightarrow E_U\rightarrow0
\end{align*}
is exact. The Jacquet module of $\pi$ with respect to $U$ is $E_U$ and we call $E(U)$ the Jacquet kernel. The normalized Jacquet module
$j_{U}(\pi)$ is defined as in \cite{BZ2} (1.8): it is the representation of $M$ on $E_U$ given by
\begin{align*}
j_{U}(m)(v+E(U))=\topmodule\nolimits_U^{-1/2}(m)(\pi(m)v+E(U)).
\end{align*}
Here $\topmodule_U$ is the topological module of $U$. If $G=\GL_n$ (or its cover) and $U$ is the unipotent radical of a parabolic subgroup $Q$,
$\topmodule_U=\delta_Q$.

We recall from \cite{BZ1} (2.32-2.33) that for any unipotent subgroups $U$ and $V$, $E(U)(V)=E(U)\cap E(V)$, and if $V\triangleleft UV$, $E(U)+E(V)=E(UV)$ and $E_{UV}=(E_{V})_U$.

Let $Q<\GL_n$ be a closed subgroup containing $N_n$ and let $E$ be a $Q$-module.
For $0\leq m\leq n$ and $b\in\{0,1\}$, define the following functor
\begin{align*}
\mathcal{L}^{n,m}_b:\Alg{(Q\cap Q_{n-m,m})}\rightarrow\Alg{(Q\cap Q_{n-m,m})}
\end{align*}
by
\begin{align*}
\mathcal{L}^{n,m}_b(E)=\begin{dcases}
E(U_{n-m,m})&b=0,\\
E_{U_{n-m,m}}&b=1.
\end{dcases}
\end{align*}
The functor $\mathcal{L}^{n,m}_b$ is exact.
Note that $\mathcal{L}^{n,n}_0(E)=0$ and $\mathcal{L}^{n,n}_1(E)=E$, because $U_{n,0}=U_{0,n}=\{I_n\}$.

Now for $b=(b_1,\ldots,b_{l})\in\{0,1\}^{l}$, where $1\leq l\leq n-m+1$,
\begin{align*}
\mathcal{L}^{n,m}_b:\Alg{(Q\cap Q_{n-m-l+1,1^{l-1},m})}\rightarrow\Alg{(Q\cap Q_{n-m-l+1,1^{l-1},m})}
\end{align*}
is given by
\begin{align*}
\mathcal{L}^{n,m}_{b}(E)=\mathcal{L}^{n,m+l-1}_{b_{l}}\ldots\mathcal{L}^{n,m+1}_{b_{2}}\mathcal{L}^{n,m}_{b_{1}}(E).
\end{align*}
The representation $\mathcal{L}^{n,1}_{b}(E)$ is always a $Q\cap B_n$-module. For formal reasons, if $l=0$ (i.e., $b$ is empty), set $\mathcal{L}^{n,m}_b(E)=E$.

The aforementioned properties of Jacquet modules imply the following lemma and corollary.
\begin{lemma}\label{lemma:composition series of Jacquet of tensor}
Let $\pi$ and $\pi'$ be representations of $\widetilde{B}_n$ on the spaces $E$ and $E'$ (resp.).
As $B_n$-modules
\begin{align*}
(E\otimes E')_{N_n}=\mathrm{s.s.}\bigoplus_{b\in\{0,1\}^{n-1}}(\mathcal{L}^{n,1}_b(E)\otimes\mathcal{L}^{n,1}_b(E'))_{N_n}.
\end{align*}
\end{lemma}
\begin{proof}[Proof of Lemma~\ref{lemma:composition series of Jacquet of tensor}]
Since
\begin{align*}
(E_{U_{n-1,1}}\otimes E'_{U_{n-1,1}})_{U_{n-1,1}}=E_{U_{n-1,1}}\otimes E'_{U_{n-1,1}}
\end{align*}
and
\begin{align*}
&(E(U_{n-1,1})\otimes E'_{U_{n-1,1}})_{U_{n-1,1}}=E(U_{n-1,1})_{U_{n-1,1}}\otimes E'_{U_{n-1,1}}=0,
\end{align*}
the module $(E\otimes E')_{U_{n-1,1}}$ is glued from
\begin{align*}
(E(U_{n-1,1})\otimes E'(U_{n-1,1}))_{U_{n-1,1}},\qquad E_{U_{n-1,1}}\otimes E'_{U_{n-1,1}}.
\end{align*}
Applying the same argument to both representations and using $U_{n-1,1},U_{n-2,2}\triangleleft U_{n-1,1}U_{n-2,2}$ and \cite{BZ1} (2.32-2.33),
\begin{align*}
(E\otimes E')_{U_{n-1,1}U_{n-2,2}}=\mathrm{s.s.}\bigoplus_{(b_1,b_2)\in\{0,1\}^2}(\mathcal{L}^{n,1}_{(b_1,b_2)}(E)\otimes\mathcal{L}^{n,1}_{(b_1,b_2)}(E'))_{U_{n-1,1}U_{n-2,2}}.
\end{align*}
Proceeding up to $U_{1,n-1}$ yields the result.
\end{proof}
\begin{corollary}\label{corollary:composition series of Jacquet of tensor}
Let $\pi$ and $\pi'$ be representations of $\widetilde{B}_n$ on the spaces $E$ and $E'$ (resp.). Assume
$0\leq m\leq n$, $0\leq l\leq n-m+1$ and $b\in\{0,1\}^l$ are given. Assume $m\geq 1$.
As $B_n$-modules
\begin{align}\label{eq:corollary composition}
(\mathcal{L}^{n,m}_b(E)\otimes\mathcal{L}^{n,m}_b(E'))_{N_n}=\mathrm{s.s.}\bigoplus_{c}(\mathcal{L}^{n,1}_c(E)\otimes\mathcal{L}^{n,1}_c(E'))_{N_n},
\end{align}
where $c=(c_1,\ldots,c_{m-1},b_1,\ldots,b_l,c_m,\ldots,c_{n-l-1})$ varies over $\{0,1\}^{n-l-1}$.

When $m=0$ we have the following special cases:
\begin{itemize}[leftmargin=*]
\item $l=0$: \eqref{eq:corollary composition} holds where $c$ varies over $\{0,1\}^{n-1}$,
\item $l>0$, $b_1=0$: $\mathcal{L}^{n,m}_b(E)=\mathcal{L}^{n,m}_b(E')=0$,
\item $l>0$, $b_1=1$: \eqref{eq:corollary composition} holds and $c=(b_2,\ldots,b_l,c_1,\ldots,c_{n-l})$ varies over $\{0,1\}^{n-l}$.
    \end{itemize}
\end{corollary}
\begin{proof}[Proof of Corollary~\ref{corollary:composition series of Jacquet of tensor}]
Assume $m\geq1$. First apply Lemma~\ref{lemma:composition series of Jacquet of tensor} to deduce
\begin{align*}
(\mathcal{L}^{n,m}_b(E)\otimes\mathcal{L}^{n,m}_b(E'))_{N_n}=\mathrm{s.s.}\bigoplus_{d\in\{0,1\}^{n-1}}
(\mathcal{L}^{n,1}_d\mathcal{L}^{n,m}_b(E)\otimes\mathcal{L}^{n,1}_d\mathcal{L}^{n,m}_b(E'))_{N_n}.
\end{align*}
We claim 
\begin{align}\label{eq:corollary filtration L}
\mathcal{L}^{n,1}_d\mathcal{L}^{n,m}_b(E)=
\begin{dcases}
\mathcal{L}^{n,1}_{d}(E)&(d_m,\ldots,d_{m+l-1})=b,\\
0&\text{otherwise.}
\end{dcases}
\end{align}
To see this we repeatedly apply the following identities, all derived from the definitions and the fact that $U_{n-k,k}\triangleleft N_n$ for all $k$. For $x,y\in\{0,1\}$ and $i\ne j$,
\begin{itemize}[leftmargin=*]
\item $\mathcal{L}^{n,i}_x\mathcal{L}^{n,j}_y=\mathcal{L}^{n,j}_y\mathcal{L}^{n,i}_x$.
\item $\mathcal{L}^{n,i}_x\mathcal{L}^{n,i}_y$ equals $\mathcal{L}^{n,i}_{x}$ if $x=y$, otherwise it vanishes.
\end{itemize}
Equality~\eqref{eq:corollary filtration L} clearly implies the result.

The remaining cases of $m=0$ follow from $\mathcal{L}^{n,0}_0(E)=0$ and $\mathcal{L}^{n,0}_1(E)=E$.
\end{proof}
\subsection{Metaplectic tensor}\label{subsection:The metaplectic tensor}
Irreducible representations of Levi subgroups of classical groups are usually described in terms of the tensor product.
Preimages in $\widetilde{\GL}_n$ of direct factors of Levi subgroups of $\GL_n$, do not commute.
Hence the tensor construction cannot be extended in a straightforward manner. The metaplectic
tensor has been studied by several authors \cite{FK,Su2,Kable,Mezo,Tk2}, in different contexts.

We briefly recall the tensor construction of Kable \cite{Kable}, whose
results will be used throughout. For a Levi subgroup $M<\GL_n$, Let $M^{\square}=\{m\in M:\det{m}\in F^{*2}\}$.
If $\pi$ is a representation of $\widetilde{M}$, denote its restriction to 
$\widetilde{M}^{\square}$ by $\pi^{\square}$.

For $i=1,2$, let $M_i<\GL_{n_i}$ be a standard Levi subgroup, regarded a subgroup of $M=M_1\times M_2<\GL_n$, $n=n_1+n_2$. Consider a pair $\pi_1$ and $\pi_2$ of genuine irreducible admissible representations of
$\widetilde{M}_1$ and $\widetilde{M}_2$.
The subgroups $\widetilde{M}_1^{\square}$ and $\widetilde{M}_2^{\square}$ commute in $\widetilde{\GL}_n$,
hence the usual tensor $\pi_1^{\square}\otimes\pi_2^{\square}$ is defined and may be regarded as a genuine representation of
$p^{-1}(M_1^{\square}\times M_2^{\square})$. In fact
$\widetilde{M}_1^{\square}$ and $\widetilde{M}_2$ also commute, the representation $\pi_1^{\square}\otimes\pi_2$
(and similarly $\pi_1\otimes\pi_2^{\square}$) is defined.

For any character $\omega$ of $\widetilde{Z}_n^{\mathe}$ which coincides with $\omega_{\pi_1}|_{\widetilde{Z}_{n_1}^2}\otimes\omega_{\pi_2}|_{\widetilde{Z}_{n_2}^2}$ on $\widetilde{Z}_n^2$, Kable \cite{Kable} defined the metaplectic tensor
$\pi_1\widetilde{\otimes}_{\omega}\pi_2$ as an irreducible summand of
\begin{align}\label{space:induced space of metaplectic tensor}
\ind_{p^{-1}(M_1^{\square}\times M_2^{\square})}^{\widetilde{M}}(\pi_1^{\square}\otimes\pi_2^{\square}),
\end{align}
on which $\widetilde{Z}_n^{\mathe}$ acts by $\omega$. The summand might not be unique, but all such summands are isomorphic (\cite{Kable} Theorem~3.1).

We mention that the definitions of Kable \cite{Kable} are more general, and include genuine
admissible finite length representations, which admit a central character. In particular for genuine admissible finite length indecomposable representations, the tensor was defined as an indecomposable summand of \eqref{space:induced space of metaplectic tensor}. When starting with irreducible representations, the tensor is irreducible (\cite{Kable} Proposition~3.3).

A more specific description was given in \cite{Kable} (Corollary~3.1): if $n_2$ is even or $n_1$ and $n_2$ are odd, there is an irreducible
summand $\sigma\subset\pi_2^{\square}$ such that
\begin{align}\label{eq:refinement for tensor}
\pi_1\widetilde{\otimes}_{\omega}\pi_2=
\ind_{p^{-1}(M_1\times M_2^{\square})}^{\widetilde{M}}(\pi_1\otimes\sigma).
\end{align}
If $n_2$ is even and $n_1$ is odd, $\sigma$ is uniquely determined by the requirement $\omega=\omega_{\pi_1}\otimes\omega_{\sigma}$ on $\widetilde{Z}_n$; if both $n_2$ and $n_1$ are even, $\sigma$ is arbitrary; otherwise both are odd and
 $\sigma=\pi_2^{\square}$. The definition for the remaining case
of odd $n_2$ and even $n_1$ is similar with the roles of $n_1$ and $n_2$ reversed.

By \cite{Kable} (Theorem~3.1),
\begin{align*}
(\pi_1\widetilde{\otimes}_{\omega}\pi_2)^{\square}=\begin{dcases}
[F^*:F^{*2}]\pi_1^{\square}\otimes\pi_2^{\square}&\text{$n_1$ and $n_2$ are odd,}\\
\pi_1^{\square}\otimes\pi_2^{\square}&\text{otherwise.}
\end{dcases}
\end{align*}
We need a slightly stronger result.
\begin{claim}\label{claim:Mackey theory applied to metaplectic tensor}
The following holds.
\begin{align*}
(\pi_1\widetilde{\otimes}_{\omega}\pi_2)|_{p^{-1}(M_1^{\square}\times M_2)}=\begin{dcases}\pi_1^{\square}\otimes\pi_2
&\text{even $n_2$,}\\
\pi_1^{\square}\otimes\bigoplus_{g\in\widetilde{M}_2^{\square}\bsl\widetilde{M}_2}\chi_g\pi_2
&\text{odd $n_1$ and $n_2$,}\\
\bigoplus_{g\in\widetilde{M}_1^{\square}\bsl\widetilde{M}_1}{\rconj{g}\sigma}\otimes\chi_g\pi_2&\text{even $n_1$, odd $n_2$.}
\end{dcases}
\end{align*}
Here $\chi_g$ is the character of $\widetilde{M}_2^{\square}\bsl\widetilde{M}_2$ given by $\chi_g(x)=(\det{x},\det{g})$ and
$\sigma$ is an irreducible summand of $\pi_1^{\square}$.
\end{claim}
\begin{remark}
By \cite{Kable} (Proposition~3.2), this claim implies the result for $(\pi_1\widetilde{\otimes}_{\omega}\pi_2)^{\square}$.
\end{remark}
\begin{proof}[Proof of Claim~\ref{claim:Mackey theory applied to metaplectic tensor}]
The assertions follow from \eqref{eq:refinement for tensor} by Mackey's theory.
For the first two cases, note that the space $(M_1^{\square}\times M_2)\bsl M/(M_1\times M_2^{\square})$ is trivial,
$\ind_{\widetilde{M}_2^{\square}}^{\widetilde{M}_2}(\sigma)=\pi_2$ when $n_2$ is even (\cite{Kable} Proposition~3.2),
and when both $n_1$ and $n_2$ are odd, $\sigma=\pi_2^{\square}$ and $\ind_{\widetilde{M}_2^{\square}}^{\widetilde{M}_2}(\sigma)=\oplus_{g}\chi_g\pi_2$, where the summation is over
$\widetilde{M}_2^{\square}\bsl\widetilde{M}_2$ (\cite{Kable} Proposition~3.1).
For the last case, we have a sum over $g\in\widetilde{M}_1^{\square}\bsl\widetilde{M}_1$ of representations
$\rconj{g}(\sigma\otimes\pi_2)$, where now $\sigma\subset\pi_1^{\square}$. Since $m_1m_2=(\det{m_1},\det{m_2})m_2m_1$ for
$m_i\in\widetilde{M}_i$, $\rconj{g}(\sigma\otimes\pi_2)=\rconj{g}\sigma\otimes\chi_{g}\pi_2$ (see \cite{Kable} p.~748).
\end{proof}

The metaplectic tensor was shown by Kable to satisfy several natural properties. For example, it is associative (\cite{Kable} Proposition~3.5). If $U_i<\GL_{n_i}$ are unipotent subgroups, $j_{U_1U_2}(\pi_1\widetilde{\otimes}_{\omega}\pi_2)=j_{U_1}(\pi_1)\widetilde{\otimes}_{\omega}j_{U_2}(\pi_2)$ (\cite{Kable} Proposition~4.1, here $j_{U_i}(\pi_i)$ might not be indecomposable). Note that in contrast with the usual tensor, it is not true in general that
$(\pi_1\widetilde{\otimes}_{\omega}\pi_2)(U_1)(U_2)=\pi_1(U_1)\widetilde{\otimes}_{\omega}\pi_2(U_2)$. Indeed, the right-hand side might not be defined (e.g., $\pi_1(U_1)$ does not necessarily admit a central character). This point complicated our proof of the ``only if" part of
Theorem~\ref{theorem:distinguished principal series} (see Proposition~\ref{proposition:only if direction, for any principal series} in Section~\ref{section:Distinguished representations}) and led to the development of some of the technical results of Section~\ref{section:geometric}.

\subsection{Exceptional representations}\label{subsection:The exceptional representations}
We describe the exceptional representations introduced and studied by Kazhdan and Patterson \cite{KP}. Recall the construction of principal series representations of $\widetilde{\GL}_n$. Let $\xi$ be a genuine character
of the center $\widetilde{T}_n^2\widetilde{Z}_n^{\mathe}$ of $\widetilde{T}_n$. We extend $\xi$ to a maximal abelian subgroup of $\widetilde{T}_n$, then
induce to a genuine representation $\rho(\xi)$ of $\widetilde{T}_n$, which is irreducible and independent of the particular extension. The corresponding principal series representation is then formed by extending $\rho(\xi)$ trivially on $N_n$, then inducing to $\widetilde{\GL}_n$.

The character $\xi$ is called exceptional if $\xi(I_{i-1},x^2,x^{-2},I_{n-i-1})=|x|$ for all $1\leq i\leq n-1$ and $x\in F^*$. In this
case the representation $\Ind_{\widetilde{B}_n}^{\widetilde{\GL}_n}(\rho(\xi))$ has a unique irreducible quotient $\theta$, called an
exceptional representation. The representation $\theta$ is admissible.

The exceptional characters $\xi$ are parameterized in the following manner. Let $\chi$ be a character of $F^*$. Let $\gamma:F^*\rightarrow\CC^*$ b a mapping such that $\gamma(xy)=\gamma(x)\gamma(y)(x,y)^{\lfloor n/2\rfloor}$ and $\gamma(x^2)=1$ for all $x,y\in F^*$. We call such a mapping a pseudo-character. Define
\begin{align}\label{eq:exceptionl character}
\xi_{\chi,\gamma}(\zeta\mathfrak{s}(zI_n)\mathfrak{s}(t))=\zeta\gamma(z)\chi(z^n\det{t})\delta_{B_n}^{1/4}(t),\qquad \zeta\in\mu_2,t\in T_n^2, z\in F^{*\mathe}.
\end{align}
Here $\mathfrak{s}:\GL_n\rightarrow\widetilde{\GL}_n$ is the section of \cite{BLS} (it is a splitting of $T_n^2$). Of course, when $n$ is even, the choice of $\gamma$ is irrelevant. When $n\equiv1\ (4)$, $\gamma$ is simply a square trivial character of $F^*$. If
$n\equiv3\ (4)$, $\gamma=\gamma_{\psi}$ for some nontrivial additive character $\psi$ of $F$. (The value of the cocycle on $(zI_n,z'I_n)$ is $(-1)^{\lfloor n/2\rfloor}$.)
The corresponding exceptional
representation will be denoted $\theta_{n,\chi,\gamma}$. Since $\chi\theta_{n,1,\gamma}=\theta_{n,\chi,\gamma}$, where on the left-hand side we regard $\chi$ as a character of $\widetilde{\GL}_n$ via $g\mapsto\chi(\det{g})$, we will occasionally set $\chi=1$. The character $\gamma$ will usually be fixed. 

The mapping $\zeta\mathfrak{s}(zI_n)\mapsto\zeta\gamma(z)$ is a genuine character of $\widetilde{Z}_n^{\mathe}$. This is precisely the central character $\omega_{\theta_{n,1,\gamma}}$.

One strong and useful property of exceptional representations, is that the Jacquet functor carries them into exceptional representations of Levi subgroups.
In particular $j_{N_n}(\theta_{n,\chi,\gamma})=\xi_{\chi,\gamma}$, in contrast with the case of general principal series representations, whose Jacquet modules with respect to $N_n$ are of length $n!$.

According to \cite{Kable} (Theorem~5.1),
\begin{align*}
&j_{U_{n_1,n_2}}(\theta_{n_1+n_2,1,\gamma})=\delta_{Q_{n_1,n_2}}^{-1/4}\theta_{n_1,1,\gamma_1}\widetilde{\otimes}_{\gamma}\theta_{n_2,1,\gamma_2},
\end{align*}
where $\gamma_1$ and $\gamma_2$ are arbitrary (nontrivial). Written without the normalization of $j_{U_{n_1,n_2}}$,
\begin{align}\label{eq:result of Kable for Jacquet module of exceptional}
(\theta_{n_1+n_2,1,\gamma})_{U_{n_1,n_2}}=\delta_{Q_{n_1,n_2}}^{1/4}\theta_{n_1,1,\gamma_1}\widetilde{\otimes}_{\gamma}\theta_{n_2,1,\gamma_2}.
\end{align}
Note that in the definition of the metaplectic tensor $\pi_1\widetilde{\otimes}_{\omega}\pi_2$ (see Section~\ref{subsection:The metaplectic tensor}), $\omega$ was a character of $\widetilde{Z}_n^{\mathe}$ which agrees with $\omega_{\pi_1}|_{\widetilde{Z}_{n_1}^2}\otimes\omega_{\pi_2}|_{\widetilde{Z}_{n_2}^2}$ on $\widetilde{Z}_n^2$. The pseudo-character $\gamma$ is regarded here as the character $\zeta\mathfrak{s}(zI_n)\mapsto\zeta\gamma(z)$.

Kazhdan and Patterson \cite{KP} (Section~I.3, see also \cite{BG} p.~145 and \cite{Kable} Theorem~5.4) proved that for $n\geq3$, if $|2|=1$ in $F$, the
exceptional representations do not have Whittaker models. For $n=3$, Flicker, Kazhdan and Savin \cite{FKS} (Lemma~6) used global methods to
extend this result to the case $|2|=1$. It is expected that arguments similar to those of \cite{FKS} will be applicable
for $n>3$ (see \cite{FKS} Lemma~6 and \cite{BG} p.~138). 

\section{Filtrations of representations induced to $Y_n$}\label{section:geometric}
In this section we compute certain filtrations of representations induced to the mirabolic subgroup. 
The results will be utilized in Section~\ref{section:Distinguished representations} for the
proof of Theorem~\ref{theorem:distinguished principal series}. Recall the functors $\Phi^+$ and $\Psi^+$ of
Bernstein and Zelevinsky \cite{BZ2}. We define analogous functors $\Phi_{\diamond}^+$ and $\Psi_{\diamond}^+$, without the normalization. For representations $\pi_0$ of $\GL_{n-2}$ and $\pi$ of $Y_{n-1}$,
\begin{align*}
&\Psi_{\diamond}^+:\Alg{\GL_{n-2}}\rightarrow\Alg{Y_{n-1}},\qquad\Phi_{\diamond}^+:\Alg{Y_{n-1}}\rightarrow\Alg{Y_{n}},\\
&\Psi_{\diamond}^+(\pi_0)=\ind_{\GL_{n-2}U_{n-2,1}}^{Y_{n-1}}(\pi_0),\qquad\Phi_{\diamond}^+(\pi)=\ind_{Y_{n-1}U_{n-1,1}}^{Y_{n}}(\pi\otimes\psi),
\end{align*}
where $\psi$ is a nontrivial additive character of $F$, considered also as a character of $U_{n-1,1}$ by
$\psi(u)=\psi(u_{n-1,n})$. In contrast with \cite{BZ2}, here the induction is not normalized.

For any $H<\GL_n$, denote $H^{\circ}=H\cap Y_n$.

\subsection{$B_n^{\circ}$-filtration of $\Phi_{\diamond}^+\Psi_{\diamond}^+$}\label{subsection:B_ncirc filtrations of phi psi}
The results of this section are stated for $Y_n$, but apply also to $\widetilde{Y}_n$. 
Note that $\Phi_{\diamond}^+\Psi_{\diamond}^+(\pi_0)=\Phi^+\Psi^+(|\det|\pi_0)$ and if $\tau$ is a representation of
$Y_n$ such that $|\det|\pi_0$ is its second derivative, $\Phi^+\Psi^+(|\det|\pi_0)$ is the second quotient appearing in the filtration of
$\tau$, with respect to its derivatives (see \cite{BZ2} 3.5). The results here make no assumption on $\pi_0$ (except being smooth).

A function $f$ in the space of $\Phi_{\diamond}^+(\pi)$ is determined by its restriction to a set of representatives
of $Y_{n-1}U_{n-1,1}\bsl Y_{n}\isomorphic F^{n-1}-\{0\}$. This isomorphism extends to a topological isomorphism, where
$F^{n-1}-\{0\}$ is regarded as an open subset of $F^{n-1}$. For $0\ne x\in F^{n-1}$, set
\begin{align*}
k_x=\min\{i:x_i\ne0\}
\end{align*}
($1\leq k_x\leq n-1$). We choose a set of representatives
$\Omega=\{\ell(x):0\ne x\in F^{n-1}\}$ as in \cite{Flicker}, with
\begin{align*}
\ell(x)=\left(\begin{array}{cccc}I_{k_x-1}&0&0&0\\0&0&I_{n-{k_x}-1}&0\\0&x_{k_x}&(x_{{k_x}+1},\ldots,x_{n-1})&0\\0&0&0&1\end{array}\right).
\end{align*}
There is a compact subset $\Omega_0\subset\Omega$ such that $f|_{\Omega}$ vanishes outside of $\Omega_0$. In particular, the image
of $f|_{\Omega}$ in the space of $\pi$ is a finite set and furthermore, there is a constant $c_f$ such that for any $\ell(x)\in\Omega_0$,
$|x_i|>c_f$ for some $i$.

We use this description to compute Jacquet modules and kernels of $\Phi_{\diamond}^+(\pi)$. In general if $U<Y_n$ is a unipotent subgroup, according to the
Jacquet-Langlands characterization of the kernel of the Jacquet functor (see e.g. \cite{BZ1} 2.33),
$\Phi_{\diamond}^+(\pi)(U)$ is the space of functions $f\in\Phi_{\diamond}^+(\pi)$, for which there is a compact subgroup $\mathcal{N}<U$, such that
\begin{align}\label{eq:Jacquet kernel as integral}
\int_{\mathcal{N}}f(\ell(x)v)\ dv=0,\qquad\forall 0\ne x\in F^{n-1}.
\end{align}

Also note that if $f\in\Phi_{\diamond}^+(\pi)$, for any $u\in U_{n-1,1}$,
\begin{align}\label{eq:action of u in geometric general n}
&u\cdot f(\ell(x))=f(\ell(x)u)=\psi(xu)f(\ell(x)).
\end{align}
Here and onward, if $x\in F^{l}$ and $u\in U_{l,1}$, when we write $\psi(xu)$ we refer to $u$ as a column in $F^l$. For example
$U_{n-1,1}=\{\left(\begin{smallmatrix}I_{n-1}&u\\&1\end{smallmatrix}\right)\}\isomorphic F^{n-1}$.

For any open $\Omega_0\subset\Omega$, denote by $\Phi_{\diamond}^{+;\Omega_0}(\pi)$ the subspace of
$\Phi_{\diamond}^+(\pi)$ consisting of functions $f$, such that the support of $f|_{\Omega}$ is contained in
$\Omega_0$. Let $\Omega(j)=\{\ell(x)\in\Omega:k_x\leq j\}$.
Then $\Phi_{\diamond}^{+;\Omega(n-1)}(\pi)=\Phi_{\diamond}^+(\pi)$. For each
$j$, $\Phi_{\diamond}^{+;\Omega(j)}(\pi)$ is a $Q_{j,1^{n-j}}^{\circ}$-module. To see this note that if $y\in Q_{j,1^{n-j}}^{\circ}$ and
$k_x>j$, $\ell(x)y=y'\ell(x')$, where $y'\in Q_{j,1^{n-j}}^{\circ}\cap Y_{n-1}U_{n-1,1}$ and $x'\in F^{n-1}$ satisfies $k_{x'}=k_x$. This follows from the computation
\begin{align*}
\left(\begin{array}{cc}0&I_l\\x&y\end{array}\right)
\left(\begin{array}{cc}a&b\\0&d\end{array}\right)
=\left(\begin{array}{cc}d&0\\0&1\end{array}\right)
\left(\begin{array}{cc}0&I_l\\ax&xb+yd\end{array}\right),\qquad x,a\in F^*.
\end{align*}
In particular we have the following filtration of $B_n^{\circ}$-modules,
\begin{align*}
0\subset\Phi_{\diamond}^{+;\Omega(1)}(\pi)\subset\ldots\subset\Phi_{\diamond}^{+;\Omega(n-1)}(\pi)=\Phi_{\diamond}^+(\pi).
\end{align*}
For formal reasons, put $\Phi_{\diamond}^{+;\Omega(0)}(\pi)=0$.
\begin{example}
If $n=4$,
\begin{align*}
&\Omega(1)=F^*\times F\times F,\\
&\Omega(2)=F^*\times F\times F\ \bigcup\ F\times F^*\times F,\\
&\Omega(3)=F^3-\{0\}.
\end{align*}
\end{example}
We will need certain generalizations of $\Phi_{\diamond}^+$. Let $2\leq i\leq j\leq n$. We define functors
\begin{align*}
\mathcal{E}_{j,i},\mathcal{E}_{j,i}^-:\Alg{Q_{n-j,1^{j-i},i-2}}\rightarrow\Alg{Q_{n-j,1^{j}}^{\circ}}.
\end{align*}
The functor $\mathcal{E}_{j,i}$ will be used to describe the quotients of the aforementioned filtration, see Claim~\ref{claim:relating phi + with exact j to mathcal E} below.
For a subgroup $X<\GL_i$ put
\begin{align*}
&E_{j,i}(X)=\left\{\left(\begin{array}{ccc}g&u_1&u_2\\&b&u_3\\&&x\end{array}\right):g\in\GL_{n-j},b\in B_{j-i},x\in X\right\}.
\end{align*}
Let $\pi_0$ be a representation of the parabolic subgroup $Q_{n-j,1^{j-i},i-2}<\GL_{n-2}$. Regarding $\GL_{n-2}$ as a subgroup
of $Y_{n-1}$, we can extend $\pi_0$ trivially on $U_{n-2,1}$ and
form a representation $\pi_0\otimes\psi$ of $E_{j,i}(Y_{i-1}U_{i-1,1})=Q_{n-j,1^{j-i},i-2}\ltimes U_{n-2,1,1}$, where $\psi$ is regarded as a character
of $U_{n-1,1}$, as above. Now consider the induced space
\begin{align*}
\ind_{E_{j,i}(Y_{i-1}U_{i-1,1})}^{E_{j,i}(Y_i)}(\pi_0\otimes\psi).
\end{align*}
This is in particular a $Q_{n-j,1^{j}}^{\circ}$-module.
Functions in this space are determined by their restriction to $Y_{i-1}U_{i-1,1}\bsl Y_{i}$. Choose a set
of representatives as above, $\Omega^{(i)}=\{\ell(x):0\ne x\in F^{i-1}\}$, then
$\Omega^{(i)}(l)=\{\ell(x)\in\Omega^{(i)}:k_x\leq l\}$. Set
\begin{align*}
\mathcal{E}_{j,i}(\pi_0)=\{f\in\ind_{E_{j,i}(Y_{i-1}U_{i-1,1})}^{E_{j,i}(Y_i)}(\pi_0\otimes\psi):\text{$f|_{\Omega^{(i)}}$ vanishes outside $\Omega^{(i)}(1)$}\}.
\end{align*}
This is still a $Q_{n-j,1^{j}}^{\circ}$-module, because for a representation $\pi_i$ of $Y_{i-1}$,
$\Phi_{\diamond}^{+;\Omega(1)}(\pi_i)$ is a $B_{i}^{\circ}$-module. Note that if $i\leq n-1$, then $Q_{0,1^{n-i},i-2}=Q_{1,1^{n-1-i},i-2}$ and
\begin{align}\label{eq:increasing by 1 when j = n-1}
\mathcal{E}_{n-1,i}(\pi_0)=\mathcal{E}_{n,i}(\pi_0).
\end{align}

By the definition, Equality~\eqref{eq:action of u in geometric general n} applies also to functions in $\mathcal{E}_{j,i}(\pi_0)$.

Further denote by $\mathcal{E}_{j,i}^-(\pi_0)$ the representation obtained by the above construction, with $\psi$ replaced by
the trivial character. The motivation for defining $\mathcal{E}_{j,i}^-$ is that by \eqref{eq:action of u in geometric general n},
\begin{align*}
(\mathcal{E}_{j,i}(\pi)\otimes\mathcal{E}_{j,i}(\pi'))_{U_{n-1,1}}=\mathcal{E}_{j,i}^-(\pi\otimes\pi').
\end{align*}
Clearly $\mathcal{E}_{j,i}$ and $\mathcal{E}_{j,i}^-$ are exact. 
The following claims will be applied repeatedly below.
\begin{claim}\label{claim:computations of Jacquet kernel and module for mathcal E}
For $0\leq m\leq\min(1,n-j)$ and $b\in\{0,1\}$, as functors
\begin{align*}
\Alg{Q_{n-j,1^{j-i},i-2}}\rightarrow\Alg{Q_{n-j-m,1^{j+m-i}}^{\circ}},
\end{align*}
\begin{align*}
&\mathcal{L}^{n,j+m}_b\mathcal{E}_{j,i}=\mathcal{E}_{j+m,i}
\mathcal{L}^{n-2,j+m-2}_b.
\end{align*}
\end{claim}
\begin{proof}[Proof of Claim~\ref{claim:computations of Jacquet kernel and module for mathcal E}]
Assume $b=0$. We need to prove that for a representation $\pi_0$,
\begin{align}\label{eq:mathcal E commutes with Jacquet kernel}
\mathcal{E}_{j,i}(\pi_0)(U_{n-j-m,j+m})=\mathcal{E}_{j+m,i}(
\pi_0(U_{n-j-m,j+m-2})).
\end{align}

First we show that $\mathcal{E}_{j,i}(\pi_0)(U_{n-j-m,j+m})$ consists of the functions $f
\in\mathcal{E}_{j,i}(\pi_0)$ such that
$f|_{\Omega^{(i)}}$ is contained in $\pi_0(U_{n-j-m,j+m-2})$.

Indeed, let $f\in \mathcal{E}_{j,i}(\pi_0)(U_{n-j-m,j+m})$. Then by \cite{BZ1} (2.33), there is a compact subgroup $\mathcal{N}<U_{n-j-m,j+m}$
for which \eqref{eq:Jacquet kernel as integral} (with $i$ instead of $n$) holds. 
As a subgroup of $E_{j,i}(Y_i)$, $Y_i$ normalizes $U_{n-l,l}$ for any $i \leq l\leq n$. Moreover,
the last two columns of $U_{n-j-m,j+m}$ act trivially on the left, because $\psi$ is trivial on
$U_{n-j-m,j+m}\cap U_{n-1,1}$ whenever $j\geq2$. Thus for all $0\ne x\in F^{i-1}$,
\begin{align*}
&0=\int_{\mathcal{N}}f(\ell(x)v)\ dv=\int_{\rconj{\ell(x)}\mathcal{N}\cap U_{n-j-m,j+m-2}}\pi_0(v)f(\ell(x))\ dv.
\end{align*}
It follows that the image of $f|_{\Omega^{(i)}}$ is contained in $\pi_0(U_{n-j-m,j+m-2})$.

Conversely, because $f|_{\Omega^{(i)}}$ is
compactly supported, one may choose a large enough compact subgroup $\mathcal{N}<U_{n-j-m,j+m}$ such that \eqref{eq:Jacquet kernel as integral} holds and hence $f\in\mathcal{E}_{j,i}(\pi_0)(U_{n-j-m,j+m})$.

It follows that restriction of $f$ to a function on $E_{j+m,i}(Y_i)$ defines an injection into the right-hand side of
\eqref{eq:mathcal E commutes with Jacquet kernel}. It is also a bijection, as we now explain (this is clear
if $m=0$).

We use the increasing filtration of $\mathcal{E}_{j+m,i}(\pi_0(U_{n-j-m,j+m-2}))$ (\cite{BZ1} 2.24, see Section~\ref{subsection:filtration of induced representations}).
Let $f_1\in\mathcal{E}_{j+m,i}(\pi_0(U_{n-j-m,j+m-2}))$. Taking a small enough compact open subgroup $\mathcal{V}_1<E_{j+m,i}(Y_i)$,
we can regard $f_1|_{\Omega^{(i)}}$ as a locally constant function on $F^{i-1}-\{0\}$ such that
\begin{align*}
f_1(\ell(x))\in\pi_0(U_{n-j-m,j+m-2})^{E_{j+m,i}(Y_{i-1}U_{i-1,1})\ \cap\ \rconj{\ell(x)}\mathcal{V}_1},\qquad\forall 0\ne x\in F^{i-1}.
\end{align*}
Select a compact open subgroup $\mathcal{V}<E_{j,i}(Y_i)$ such that each vector $f_1(\ell(x))$ is fixed by
$E_{j,i}(Y_{i-1}U_{i-1,1})\cap\rconj{\ell(x)}\mathcal{V}$. Since
\begin{align*}
E_{j,i}(Y_{i-1}U_{i-1,1})\bsl E_{j,i}(Y_i)\isomorphic E_{j+m,i}(Y_{i-1}U_{i-1,1})\bsl E_{j+m,i}(Y_i),
\end{align*}
we can use $\mathcal{V}$ to define $f\in\mathcal{E}_{j,i}(\pi_0)$ with $f|_{\Omega^{(i)}}=f_1|_{\Omega^{(i)}}$. It
follows that $f$ belongs to $\mathcal{E}_{j,i}(\pi_0)(U_{n-j-m,j+m})$ and is the preimage of $f_1$.

Regarding the case of $b=1$, according to the arguments above, if
$\alpha:\pi_0\rightarrow(\pi_0)_{U_{n-j-m,j+m-2}}$ is the natural projection,
the mapping $\mathcal{E}_{j,i}(\pi_0)\rightarrow\mathcal{E}_{j+m,i}((\pi_0)_{U_{n-j-m,j+m-2}})$ given
by $f\rightarrow \alpha f|_{E_{j+m,i}(Y_i)}$ is onto and its kernel is precisely
$\mathcal{E}_{j,i}(\pi_0)(U_{n-j-m,j+m})$.
\end{proof}
\begin{claim}\label{claim:relating phi + with exact j to mathcal E}
For $1\leq j\leq n-1$, as $Q_{j-1,1^{n-j+1}}^{\circ}$-modules
\begin{align*}
\frac{\Phi_{\diamond}^{+;\Omega(j)}\Psi_{\diamond}^+(\pi_0)}
{\Phi_{\diamond}^{+;\Omega(j-1)}\Psi_{\diamond}^+(\pi_0)}
=\mathcal{E}_{n-j+1,n-j+1}(\pi_0).
\end{align*}
\end{claim}
\begin{proof}[Proof of Claim~\ref{claim:relating phi + with exact j to mathcal E}]
Put $l=n-j+1$. First observe that under the embedding $Y_l<E_{l,l}(Y_l)<Y_n$,
\begin{align*}
\Omega^{(l)}(1)=\{\ell(x):0\ne x\in F^{n-1},k_x=j\}.
\end{align*}
Also in general, restriction of a locally constant compactly supported function on $F^*\times F\cup F\times F^*$ to
$\{0\}\times F^*$ is a locally constant compactly supported function.
Hence restriction $f\mapsto f|_{E_{l,l}(Y_l)}$ defines
a mapping
\begin{align*}
\Phi_{\diamond}^{+;\Omega(j)}\Psi_{\diamond}^+(\pi_0)\rightarrow
\mathcal{E}_{l,l}(\pi_0)
\end{align*}
whose kernel is $\Phi_{\diamond}^{+;\Omega(j-1)}\Psi_{\diamond}^+(\pi_0)$.

To show this is onto, we argue as in the proof of Claim~\ref{claim:computations of Jacquet kernel and module for mathcal E}.
For $f_1\in\mathcal{E}_{l,l}(\pi_0)$, regard $f_1|_{\Omega^{(l)}}$ as a locally constant function, whose support is contained in
 $\Omega^{(l)}(1)$, and for $0\ne x\in F^{i-1}$,
\begin{align*}
f_1(\ell(x))\in\pi_0^{Y_{n-1}U_{n-1,1}\ \cap\ \rconj{\ell(x)}\mathcal{V}}.
\end{align*}
Here $\mathcal{V}<Y_n$ is compact open (in particular $\mathcal{V}\cap E_{l,l}(Y_l)$ is compact open).
Since $\Omega^{(l)}(1)$ is a closed subgroup of $F^{n-l}\times\Omega^{(l)}(1)$, and the latter is open in $\Omega$,
there is a locally constant function on $\Omega$, such that its restriction to $\Omega^{(l)}$ agrees
with $f_1|_{\Omega^{(l)}}$. Hence we can define $f\in \Phi_{\diamond}^{+;\Omega(j)}\Psi_{\diamond}^+(\pi_0)$ satisfying
$f|_{\Omega^{(l)}}=f_1|_{\Omega^{(l)}}$.
\end{proof}
Here is the main result of this section.
\begin{lemma}\label{lemma:L of Y_n module}
Let $\pi_0$ be a representation of $\GL_{n-2}$, $n\geq3$, and let $b\in\{0,1\}^{n-2}$.
If $b=0^{n-2}$, set $k=n-2$; if $b_1=1$, set $k=0$; otherwise let $k$ be the first index such that $b_k=0$ and $b_{k+1}=1$.
As a $B_n^{\circ}$-module
\begin{align*}
\mathcal{L}^{n,2}_{b}(\Phi_{\diamond}^+\Psi_{\diamond}^+(\pi_0))=\mathrm{s.s.}\bigoplus_{m=\min(1,k)}^k
\mathcal{E}_{n,m+2}
\mathcal{L}^{n-2,m}_{(0^{k-m},b_{k+1},\ldots,b_{n-2})}(\pi_0).
\end{align*}
\end{lemma}
\begin{proof}[Proof of Lemma~\ref{lemma:L of Y_n module}]
We prove the lemma in three steps. 

\begin{claim}\label{claim:intersection of Jacquet modules up to k}
For $1\leq k\leq n-2$, as a $Q_{n-k-1,1^{k+1}}^{\circ}$-module $\mathcal{L}^{n,2}_{0^k}\Phi_{\diamond}^+\Psi_{\diamond}^+(\pi_0)$ is glued from
\begin{align*}
&\Phi_{\diamond}^{+;\Omega(n-k-1)}\Psi_{\diamond}^+(\pi_0),\\
&\mathcal{E}_{k+1,m+2}\mathcal{L}^{n-2,m}_{0^{k-m}}(\pi_0),\qquad 1\leq m\leq k-1.
\end{align*}
\end{claim}

\begin{claim}\label{claim:co-invariant at k+1 of intersection of Jacquet modules}
If $0\leq k\leq n-3$, as $Q_{n-k-2,1^{k+2}}^{\circ}$-modules
\begin{align*}
\mathcal{L}^{n,2}_{(0^k,1)}\Phi_{\diamond}^+\Psi_{\diamond}^+(\pi_0)=\mathrm{s.s.}\bigoplus_{m=\min(1,k)}^k
\mathcal{E}_{k+2,m+2}
\mathcal{L}^{n-2,m}_{(0^{k-m},1)}(\pi_0).
\end{align*}
\end{claim}

\begin{claim}\label{claim:rest of b Jacquet module after b_1=0}
For $0\leq k\leq n-4$ and $\min(1,k)\leq m\leq k$, 
as $B_n^{\circ}$-modules
\begin{align*}
\mathcal{L}^{n,k+3}_{(b_{k+2},\ldots,b_{n-2})}\mathcal{E}_{k+2,m+2}(
\mathcal{L}^{n-2,m}_{(0^{k-m},1)}(\pi_0))
=\mathcal{E}_{n,m+2}
\mathcal{L}^{n-2,m}_{(0^{k-m},1,b_{k+2},\ldots,b_{n-2})}(\pi_0).
\end{align*}
\end{claim}

The lemma follows from these claims (proved below). Specifically,
for $k=n-2$ apply Claim~\ref{claim:intersection of Jacquet modules up to k}, use
$\Phi_{\diamond}^{+;\Omega(1)}\Psi_{\diamond}^+(\pi_0)=\mathcal{E}_{n,n}(\pi_0)$
and \eqref{eq:increasing by 1 when j = n-1} for $1\leq m\leq n-3$.
If $k=n-3$, the result is stated in Claim~\ref{claim:co-invariant at k+1 of intersection of Jacquet modules} (and again use
\eqref{eq:increasing by 1 when j = n-1}).
For $k\leq n-4$ apply Claims~\ref{claim:co-invariant at k+1 of intersection of Jacquet modules} and \ref{claim:rest of b Jacquet module after b_1=0}, note that
\begin{align*}
\mathcal{L}^{n,2}_{b}=\mathcal{L}^{n,k+3}_{(b_{k+2},\ldots,b_{n-2})}\mathcal{L}^{n,2}_{(0^k,1)}
\end{align*}
and these functors are exact.
\begin{proof}[Proof of Claim~\ref{claim:intersection of Jacquet modules up to k}]
By definition
\begin{align*}
\mathcal{L}^{n,2}_{0^k}\Phi_{\diamond}^+\Psi_{\diamond}^+(\pi_0)=\bigcap_{m=2}^{k+1}\Phi_{\diamond}^+\Psi_{\diamond}^+(\pi_0)(U_{n-m,m}).
\end{align*}
Let $2\leq l\leq n-1$. According to Claim~\ref{claim:relating phi + with exact j to mathcal E} (with $j=n-l+1$) and the exactness of taking a Jacquet kernel, as a $Q_{n-l,1^l}^{\circ}$-module $\Phi_{\diamond}^{+;\Omega(n-l+1)}(U_{n-l,l})$ is glued from
\begin{align*}
\Phi_{\diamond}^{+;\Omega(n-l)}\Psi_{\diamond}^+(\pi_0)(U_{n-l,l}),\qquad
\mathcal{E}_{l,l}(\pi_0)(U_{n-l,l}).
\end{align*}

If $u\in U_{n-l,l}\cap U_{n-1,1}$ and $k_x\leq n-l$,
the character appearing on the right-hand side of \eqref{eq:action of u in geometric general n} is a nontrivial function of $u$. Hence
\begin{align*}
\Phi_{\diamond}^{+;\Omega(n-l)}\Psi_{\diamond}^+(\pi_0)(U_{n-l,l})
=\Phi_{\diamond}^{+;\Omega(n-l)}\Psi_{\diamond}^+(\pi_0).
\end{align*}
Also by Claim~\ref{claim:computations of Jacquet kernel and module for mathcal E} (with $j=l$ and $m=b=0$)
\begin{align*}
\mathcal{E}_{l,l}(\pi_0)(U_{n-l,l})=\mathcal{E}_{l,l}(\pi_0(U_{n-l,l-2})).
\end{align*}
Hence $\Phi_{\diamond}^{+;\Omega(n-l+1)}(U_{n-l,l})$ is glued from
\begin{align*}
\Phi_{\diamond}^{+;\Omega(n-l)}\Psi_{\diamond}^+(\pi_0),\qquad
\mathcal{E}_{l,l}(\pi_0(U_{n-l,l-2})).
\end{align*}
The result follows from a repeated application of this observation 
and Claim~\ref{claim:computations of Jacquet kernel and module for mathcal E}.
\end{proof}

\begin{proof}[Proof of Claim~\ref{claim:co-invariant at k+1 of intersection of Jacquet modules}]
Since $\mathcal{L}^{n,2}_{(0^k,1)}=\mathcal{L}^{n,k+2}_{1}\mathcal{L}^{n,2}_{0^k}$ and these functors are exact,
we can apply $\mathcal{L}^{n,k+2}_{1}$ to the factors computed by Claim~\ref{claim:intersection of Jacquet modules up to k}.
As explained in the proof of Claim~\ref{claim:intersection of Jacquet modules up to k} (with $l=k+2$, $2\leq l\leq n-1$ because
$0\leq k\leq n-3$),
\begin{align*}
\Phi_{\diamond}^{+;\Omega(n-k-2)}\Psi_{\diamond}^+(\pi_0)=\mathcal{L}^{n,k+2}_{0}\Phi_{\diamond}^{+;\Omega(n-k-2)}\Psi_{\diamond}^+(\pi_0). \end{align*}
Hence
\begin{align*}
\mathcal{L}^{n,k+2}_{1}\Phi_{\diamond}^{+;\Omega(n-k-2)}\Psi_{\diamond}^+(\pi_0)=0
\end{align*}
and
\begin{align*}
\mathcal{L}^{n,k+2}_{1}\Phi_{\diamond}^{+;\Omega(n-k-1)}\Psi_{\diamond}^+(\pi_0)
&=\mathcal{L}^{n,k+2}_{1}\left(\frac{\Phi_{\diamond}^{+;\Omega(n-k-1)}\Psi_{\diamond}^+(\pi_0)}{\Phi_{\diamond}^{+;\Omega(n-k-2)}\Psi_{\diamond}^+(\pi_0)}\right)\\
&=\mathcal{L}^{n,k+2}_{1}\mathcal{E}_{k+2,m+2}(\pi_0)=\mathcal{E}_{k+2,m+2}\mathcal{L}^{n-2,k}_{1}(\pi_0).
\end{align*}
Here the second equality follows from Claim~\ref{claim:relating phi + with exact j to mathcal E} and the third from
Claim~\ref{claim:computations of Jacquet kernel and module for mathcal E}. Note that for $k\leq 1$ this already gives the result,
because then $k=m$.

For $1\leq m\leq k-1$, applying Claim~\ref{claim:computations of Jacquet kernel and module for mathcal E}
and using $\mathcal{L}^{n-2,k}_1\mathcal{L}^{n-2,m}_{0^{k-m}}=\mathcal{L}^{n-2,m}_{(0^{k-m},1)}$,
\begin{align*}
\mathcal{L}^{n,k+2}_1\mathcal{E}_{k+1,m+2}(\mathcal{L}^{n-2,m}_{0^{k-m}}(\pi_0))
=\mathcal{E}_{k+2,m+2}\mathcal{L}^{n-2,m}_{(0^{k-m},1)}(\pi_0).
\end{align*}
The result follows.
\end{proof}

\begin{proof}[Proof of Claim~\ref{claim:rest of b Jacquet module after b_1=0}]
This follows from a repeated application of Claim~\ref{claim:computations of Jacquet kernel and module for mathcal E}.
\end{proof}
\end{proof}

\subsection{Jacquet modules of $\mathcal{E}_{n,i}^-(\pi_0)$}\label{subsection:Jacquet modules of mathcal E - n i}
Recall the functor
\begin{align*}
\mathcal{E}_{n,i}^-:\Alg{Q_{1^{n-i},i-2}}\rightarrow\Alg{B_{n}^{\circ}}
\end{align*}
defined in Section~\ref{subsection:B_ncirc filtrations of phi psi}. Assume $n\geq3$. We compute $\mathcal{L}^{n,1}_{1^{n-1}}\mathcal{E}_{n,i}^-$.

We describe an operation of a partial Fourier transform on
$\mathcal{E}_{n,i}^-(\pi_0)$. Let $\mathcal{S}(\Omega^{(i)}(1),\pi_0)$ be the space of Schwartz-Bruhat functions
on $\Omega^{(i)}(1)\isomorphic F^*\times F^{i-2}$ taking values in the space of $\pi_0$. Let $f\in\mathcal{E}_{n,i}^-(\pi_0)$. Then $f_{\triangle}=f|_{\Omega^{(i)}(1)}\in\mathcal{S}(\Omega^{(i)}(1),\pi_0)$ (see Section~\ref{subsection:filtration of induced representations}). Consider the partial Fourier transform of $f_{\triangle}$,
\begin{align*}
\widehat{f_{\triangle}}(\ell(x))=\int_{F^{i-2}}f_{\triangle}(\ell(x_1,y))\psi^{-1}((x_2,\ldots,x_{i-1})\transpose{y})dy.
\end{align*}
Here $\transpose{y}$ denotes the transpose of the row $y$. Then $\widehat{f_{\triangle}}\in\mathcal{S}(\Omega^{(i)}(1),\pi_0)$. The
mapping $f\mapsto \widehat{f_{\triangle}}$ is a linear embedding of $\mathcal{E}_{n,i}^-(\pi_0)$ in $\mathcal{S}(\Omega^{(i)}(1),\pi_0)$, whose image is denoted $\widehat{\mathcal{E}}_{n,i}^-(\pi_0)$. This embedding is extended to an embedding of $B_n^{\circ}$-modules by defining $b\cdot\widehat{f_{\triangle}}=\widehat{(bf)_{\triangle}}$. When $i=2$, the Fourier transform is trivial and $\widehat{f_{\triangle}}=f_{\triangle}$.

We will need more explicit formulas, for the $B_n^{\circ}$ action on $\widehat{\mathcal{E}}_{n,i}^-(\pi_0)$ in a few cases. First observe that for any $\ell(x)\in\Omega^{(i)}(1)$, $1\leq l\leq i-2$ and $u\in U_{l,1}<\GL_{l+1}<Y_i$,
\begin{align}\label{eq:multiplication s_n(x) b u}
\ell(x)u=u_0\ell(x+(0_l,\sum_{m=1}^lx_mu_m,0_{i-2-l})).
\end{align}
Here $u_0$ is the element in $U_{l-1,1}$, corresponding to
the column obtained from $u$ by removing its first coordinate.
\begin{example}
When $i=n=4$,
\begin{align*}
&\left(\begin{array}{cccc}&1\\&&1\\x_1&x_2&x_3&\\&&&1\end{array}\right)
\left(\begin{array}{cccc}1&v_1&v_2\\&1&v_3\\&&1&\\&&&1\end{array}\right)
=\left(\begin{array}{cccc}1&v_3\\&1&\\&&1&\\&&&1\end{array}\right)
\left(\begin{array}{cccc}&1\\&&1\\x_1&x_2+x_1v_1&x_3+x_1v_2+x_2v_3&\\&&&1\end{array}\right).
\end{align*}
\end{example}
Now for $0\ne x\in F^{i-1}$, $u\in U_{n-i+l,1}\cap Y_i$ with $1\leq l\leq i-2$, $t=diag(I_{n-i},t_1,\ldots,t_{i-1},1)$ and $b\in B_{n-i}$,
\begin{align}\label{eq:multiplication s_n(x) b u after Fourier}
&u\cdot\widehat{f_{\triangle}}(\ell(x))=
\psi(x_{l+1}(x_1,\ldots,x_{l})u)\ \pi_0(u_0)\ \widehat{f_{\triangle}}(\ell(x)),\\
\label{eq:multiplication t after Fourier}
&t\cdot\widehat{f_{\triangle}}(\ell(x))=\prod_{j=2}^{i-1}|t_{j}|^{-1}\ \pi_0(diag(I_{n-i},t_2,\ldots,t_{i-1}))\ \widehat{f_{\triangle}}
(\ell(t_{1}x_1,t_{2}^{-1}x_2,\ldots,t_{i-1}^{-1}x_{i-1})),\\
\label{eq:multiplication b after Fourier}
&b\cdot\widehat{f_{\triangle}}(\ell(x))=\pi_0(b)\widehat{f_{\triangle}}(\ell(x)).
\end{align}
Here $u_0$ is defined as above and if $i=2$,
$diag(I_{n-i},t_2,\ldots,t_{i-1})=I_{n-2}$. Equality~\eqref{eq:multiplication s_n(x) b u after Fourier} follows from
\eqref{eq:multiplication s_n(x) b u}; \eqref{eq:multiplication b after Fourier} holds because $Y_i$ commutes with $B_{n-i}$. 

\begin{lemma}\label{lemma:Jacquet functor without character}
Regard $(\pi_0)_{N_{n-2}}$ as a $T_{n-2}$-module. Then as $T_n^{\circ}$-modules,
\begin{align*}
\mathcal{E}_{n,i}^-(\pi_0)_{N_{n-1}}\isomorphic\ind_{\rconj{w_i}T_{n-2}}^{T_n^{\circ}}(\delta_{Y_i}^{-1}\ \rconj{w_i}((\pi_0)_{N_{n-2}})),
\qquad w_i=diag(I_{n-i},\left(\begin{smallmatrix}&1\\I_{i-1}\end{smallmatrix}\right)).
\end{align*}
\end{lemma}
\begin{remark}
Note that $T_n^{\circ}=T_{n-1}$ and $\rconj{w_i}T_{n-2}=\{diag(t_1,\ldots,t_{n-i},1,t_{n-i+1},\ldots,t_{n-2},1)\}$.
\end{remark}
\begin{proof}[Proof of Lemma~\ref{lemma:Jacquet functor without character}]
Regard $N_{i-1}<\GL_{i-1}$ as a subgroup of $Y_i$ (embedded in $E_{n,i}(Y_i)$). Then
\begin{align*}
\mathcal{E}_{n,i}^-(\pi_0)_{N_{n-1}}=
(\mathcal{E}_{n,i}^-(\pi_0)_{U_{1^{n-i},i-1}})_{N_{i-1}}.
\end{align*}
Since $U_{n-1,1}$ acts trivially on $\mathcal{E}_{n,i}^-(\pi_0)$,
$\mathcal{E}_{n,i}^-(\pi_0)_{U_{1^{n-i},i-1}}=\mathcal{E}_{n,i}^-(\pi_0)_{U_{1^{n-i},i}}$.
Observe that
\begin{align}\label{eq:Fourier claim starting identity}
\mathcal{E}_{n,i}^-(\pi_0)_{U_{1^{n-i},i}}=\mathcal{E}_{n,i}^-((\pi_0)_{U_{1^{n-i},i-2}}).
\end{align}
Indeed, similarly to the proof of Claim~\ref{claim:computations of Jacquet kernel and module for mathcal E}, since
$Y_i$ normalizes $U_{1^{n-i},i}$ and $\pi_0$ is trivial on the last two columns of $U_{1^{n-i},i}$ ($i\geq 2$),
\begin{align*}
(\mathcal{E}_{n,i}^-(\pi_0))(U_{1^{n-i},i})=\mathcal{E}_{n,i}^-(\pi_0(U_{1^{n-i},i-2})).
\end{align*}
Then \eqref{eq:Fourier claim starting identity} holds because $(\pi_0)_{U_{1^{n-i},i-2}}$ is still a representation of $Q_{1^{n-i},i-2}$ and $\mathcal{E}^-_{n,i}$ is exact.

Put $\vartheta=(\pi_0)_{U_{1^{n-i},i-2}}$, then
\begin{align*}
\mathcal{E}_{n,i}^-(\pi_0)_{N_{n-1}}\isomorphic
\widehat{\mathcal{E}}_{n,i}^-(\vartheta)_{N_{i-1}}.
\end{align*}

Next we claim
\begin{align}\label{eq:Fourier claim second identity}
(\widehat{\mathcal{E}}_{n,i}^-(\vartheta))(N_{i-1})=\{\widehat{f_{\triangle}}\in\widehat{\mathcal{E}}_{n,i}^-(\vartheta):\widehat{f_{\triangle}}(\ell(x_1,0^{i-2}))\in\vartheta(N_{i-2}),\forall x_1\in F^*\}.
\end{align}
This is trivial when $i=2$ (both sides are equal to zero), assume $i>2$.
If $\widehat{f_{\triangle}}\in (\widehat{\mathcal{E}}_{n,i}^-(\vartheta))(N_{i-1})$, Equality~\eqref{eq:Jacquet kernel as integral} implies that for some compact subgroup $\mathcal{N}<N_{i-1}$,
\begin{align}\label{eq:Jacquet integral vanishes on 0^i-2}
\int_{\mathcal{N}}v\cdot \widehat{f_{\triangle}}(\ell(x_1,0^{i-2}))\ dv=0,\qquad \forall x_1\in F^*.
\end{align}
Since \eqref{eq:multiplication s_n(x) b u after Fourier} shows
$v\cdot \widehat{f_{\triangle}}(\ell(x_1,0^{i-2}))=\vartheta(v_0)\widehat{f_{\triangle}}(\ell(x_1,0^{i-2}))$ and if $v$ varies in
a large compact subgroup of $N_{i-1}$, $v_0$ varies in a large compact subgroup of $N_{i-2}$, 
we see that $\widehat{f_{\triangle}}$ belongs to the right-hand side of \eqref{eq:Fourier claim second identity}.

In the other direction, let $\mathcal{N}<N_{i-1}$ be compact such that
\eqref{eq:Jacquet integral vanishes on 0^i-2} holds. Let $m_0$ be such that the support of $\widehat{f_{\triangle}}$ in $x_1$ is contained in $\{x\in F^*:q^{-m_0}<|x|<q^{m_0}\}$, where $q$ is the residual characteristic of the field. Select a large enough $m$, with respect to $\widehat{f_{\triangle}}$, $\mathcal{N}$ and $m_0$, such that if $q^{-m_0}<|x_1|<q^{m_0}$, $|x_2|\leq q^{-m},\ldots,|x_{i-1}|\leq q^{-m}$, then
\begin{align*}
&\widehat{f_{\triangle}}(\ell(x_1,\ldots,x_{i-1}))=
\widehat{f_{\triangle}}(\ell(x_1,0^{i-2})),\\
&\psi(x_{l+1}(x_1,\ldots,x_l)v_i)=1,\qquad\forall v=v_1\ldots v_{i-2}\in\mathcal{N},\ v_i\in U_{n-i+l,1},\
1\leq l\leq i-2.
\end{align*}
Note that if $v\in\mathcal{N}$, we can write uniquely $v=v_1\ldots v_{i-2}$ with $v_i\in U_{n-i+l,1}$, then the coordinates of $v_i$ are bounded from above by a constant depending on $\mathcal{N}$ and $i$.
Next take a compact $\mathcal{N}<\mathcal{N}_1<N_{i-1}$ such that if $q^{-m_0}<|x_1|<q^{m_0}$ and $|x_{l+1}|>q^{-m}$ for some $1\leq l\leq i-2$,
\begin{align*}
\int\psi(x_{l+1}x_1a)\ da=0,
\end{align*}
where $a$ varies over any nontrivial coordinate of $\mathcal{N}_1$.

We show that for all $0\ne x\in F^{i-1}$,
\begin{align}\label{eq:Jacquet integral vanishes on 0^i-2 2}
\int_{\mathcal{N}_1}v\cdot\widehat{f_{\triangle}}(\ell(x))\ dv=0.
\end{align}

Observe that by
\eqref{eq:multiplication s_n(x) b u after Fourier},
if $\widehat{f_{\triangle}}(\ell(x))=0$, $v\cdot\widehat{f_{\triangle}}(\ell(x))=0$ for all $v\in N_{i-1}$. Hence
\eqref{eq:Jacquet integral vanishes on 0^i-2 2} holds unless $q^{-m_0}<|x_1|<q^{m_0}$. If $|x_2|,\ldots,|x_{i-1}|\leq q^{-m}$,
Equality~\eqref{eq:multiplication s_n(x) b u after Fourier} implies
\begin{align*}
v\cdot\widehat{f_{\triangle}}(\ell(x))=v\cdot\widehat{f_{\triangle}}(\ell(x_1,0^{i-2})),\qquad\forall v\in\mathcal{N}.
\end{align*}
Hence by \eqref{eq:Jacquet integral vanishes on 0^i-2},
\begin{align*}
\int_{\mathcal{N}}v\cdot\widehat{f_{\triangle}}(\ell(x))\ dv=0
\end{align*}
and \eqref{eq:Jacquet integral vanishes on 0^i-2 2} follows.
Otherwise $|x_{l+1}|>q^{-m}$ for some $1\leq l\leq i-2$ and then
\begin{align*}
\int_{\mathcal{N}_1\cap U_{n-i+l,1} \cap U_{n-i+1,i-1}}v\cdot\widehat{f_{\triangle}}(\ell(x))\ dv=\widehat{f_{\triangle}}(\ell(x))
\int_{\mathcal{N}_1\cap U_{n-i+l,1} \cap U_{n-i+1,i-1}}\psi(x_{l+1}x_1v)\ dv=0.
\end{align*}
We conclude that \eqref{eq:Jacquet integral vanishes on 0^i-2 2} holds for all $x$.

Now consider the function
\begin{align*}
f_1(t)=t\cdot \widehat{f_{\triangle}}(\ell(1,0^{i-2}))+\vartheta(N_{i-2}),\qquad t\in T_n^{\circ}.
\end{align*}
Equalities~\eqref{eq:multiplication t after Fourier}, \eqref{eq:multiplication b after Fourier} and \eqref{eq:Fourier claim second identity}
imply that $\widehat{f_{\triangle}}\mapsto f_1$ is an isomorphism between
$\widehat{\mathcal{E}}_{n,i}^-(\vartheta)_{N_{i-1}}$ and $\ind_{\rconj{w_i}T_{n-2}}^{T_n^{\circ}}(\delta_{Y_i}^{-1}\ \rconj{w_i}(\vartheta_{N_{i-2}}))$. The latter is isomorphic to
$\ind_{\rconj{w_i}T_{n-2}}^{T_n^{\circ}}(\delta_{Y_i}^{-1}\ \rconj{w_i}((\pi_0)_{N_{n-2}}))$.
\end{proof}

\section{Distinguished representations}\label{section:Distinguished representations}
\subsection{Definitions}\label{subsection:definitions}
Let $\tau$ be an admissible representation of $\GL_n$ with a central character $\omega_{\tau}$. Let $\chi$ and $\chi'$ be characters of $F^*$ and let $\gamma$ and $\gamma'$ be a pair of pseudo-characters (see Section~\ref{subsection:The exceptional representations}). We say that $\tau$ is $(\chi,\gamma,\chi',\gamma')$-distinguished if
\begin{align*}
\Hom_{\GL_n}(\theta_{n,\chi,\gamma}\otimes\theta_{n,\chi',\gamma'},\tau^{\vee})\ne0.
\end{align*}
Here $\theta_{n,\chi,\gamma}\otimes\theta_{n,\chi',\gamma'}$ is regarded as a representation of $\GL_n$, as explained in Section~\ref{subsection:representations}.
Equivalently, the space
\begin{align*}
\Tri_{\GL_n}(\tau,\theta_{n,\chi,\gamma},\theta_{n,\chi',\gamma'})
\end{align*}
of $\GL_n$-invariant trilinear forms on $\tau\times\theta_{n,\chi,\gamma}\times\theta_{n,\chi',\gamma'}$ is nonzero.

If $\tau$ is $(\chi,\gamma,\chi',\gamma')$-distinguished, in particular $\omega_{\tau}^{-1}=\theta_{n,\chi,\gamma}\theta_{n,\chi',\gamma'}$ on $Z_n^{\mathe}$, and according to \eqref{eq:exceptionl character},
\begin{align}\label{eq:distinguished centeral characters condition}
\omega_{\tau}^{-1}(z I_n)=\chi(z^n)\chi'(z^n)\gamma(z)\gamma'(z),\qquad \forall z\in F^{*\mathe}.
\end{align}

The next claim implies that the appearance of the pseudo-characters in the definition is redundant and furthermore, if we fix $\gamma$, $\gamma'$ will be determined by \eqref{eq:distinguished centeral characters condition}.
\begin{claim}\label{claim:changing the character psi}
Assume that $\tau$ is $(\chi,\gamma_0,\chi',\gamma_0')$-distinguished. Then for any $\gamma$, Equality~\eqref{eq:distinguished centeral characters condition} determines
a set of pseudo-characters $\gamma'$ such that $\tau$ is $(\chi,\gamma,\chi',\gamma')$-distinguished.
\end{claim}
\begin{proof}[Proof of Claim~\ref{claim:changing the character psi}]
This is trivial when $n$ is even. In the odd case let $\gamma_0$ be given. Then $\eta=\gamma/\gamma_0$ is a square trivial character of $F^*$. Define $\gamma'=\eta\gamma_0'$, it is a pseudo-character. Then since $\eta(z)=\eta^{n}(z)$ ($n$ is odd),
\begin{align*}
\theta_{n,\chi,\gamma}=\theta_{n,\chi,\eta\gamma_0}=\theta_{n,\eta\chi,\gamma_0}=\eta\theta_{n,\chi,\gamma_0}
\end{align*}
and
\begin{align*}
\theta_{n,\chi,\gamma}\otimes\theta_{n,\chi',\gamma'}=
\eta^2\theta_{n,\chi,\gamma_0}\otimes\theta_{n,\chi',\gamma_0'}=
\theta_{n,\chi,\gamma_0}\otimes\theta_{n,\chi',\gamma_0'}.
\end{align*}
This implies that $\tau$ is $(\chi,\gamma,\chi',\gamma')$-distinguished.
\end{proof}

Additionally, because $\theta_{n,\chi,\gamma}=\chi\theta_{n,1,\gamma}$, we see that $\tau$ is $(\chi,\gamma,\chi',\gamma')$-distinguished if and only if $\tau_0=\chi\chi'\cdot\tau$ is $(1,\gamma,1,\gamma')$-distinguished. Thus we can simply take $\chi=\chi'=1$.

In light of the observations above, we say that
$\tau$ is distinguished if it is $(1,\gamma,1,\gamma')$-distinguished for some $\gamma$ and $\gamma'$.

As an example consider the minimal case of $n=1$.
\begin{claim}\label{claim:distinguished for n=1}
For $n=1$, $\tau$ is distinguished if and only if $\tau^2=1$.
\end{claim}
\begin{proof}[Proof of Claim~\ref{claim:distinguished for n=1}]
Assume that $\tau$ is distinguished. Then $\tau^2=1$ follows from \eqref{eq:distinguished centeral characters condition} because $\gamma^2{\gamma'}^2=\gamma^4=1$ ($\gamma'/\gamma$ is a square trivial character hence ${\gamma'}^2=\gamma^2$, $\gamma^4=1$ because
$\gamma(x^4)=\gamma(x^2)\gamma(x^2)(x^2,x^2)$).
Conversely, for a fixed 
$\gamma$, $z\mapsto\gamma(z)\gamma'(z)$ is a character of $F^{*2}\bsl F^*$ and any such character is obtained by varying $\gamma'$. Hence $\omega_{\tau}^{-1}=\gamma\gamma'$ for some $\gamma'$ and $\tau$ is distinguished.
\end{proof}
The proof implies the following observation, for any $n$ (already noted in \cite{Kable} p.~766).
\begin{corollary}\label{corollary:central character of distinguished}
If $\tau$ is distinguished, $\omega_{\tau}^2=1$.
\end{corollary}

\subsection{Heredity}\label{subsection:heredity}
We now prove the upper heredity of distinguished representations. The proof is based on the following observation: given an invariant form on a representation $\xi$ of a Levi subgroup, one can construct an invariant form on a representation parabolically induced from $\xi$ (assuming a certain condition on modulus characters). One complication here, is that we do not have invariancy with respect to a full Levi subgroup. We compensate for this using the properties of the metaplectic tensor product (see Section~\ref{subsection:The metaplectic tensor}).
\begin{proof}[Proof of Theorem~\ref{theorem:upper heredity of dist}]
Let $\tau_1$ and $\tau_2$ be a pair of distinguished representations of $\GL_{n_1}$ and $\GL_{n_2}$.
We have to prove that $\tau=\Ind_{Q_{n_1,n_2}}^{\GL_n}(\tau_1\otimes\tau_2)$ is distinguished, with $n=n_1+n_2$.

We may assume that either $n_1$ and $n_2$ have the same parity, or $n_2$ is even (if $n_1$ is even and $n_2$ is odd, the argument will be repeated with their roles replaced).
Let $Q_{n_1,n_2}^{\star}=M_{n_1,n_2}^{\star}\ltimes U_{n_1,n_2}$, where
\begin{align*}
M_{n_1,n_2}^{\star}=\GL_{n_1}\times\GL_{n_2}^{\square}.
\end{align*}
The assumptions on $\tau_i$ imply that for suitable pairs of pseudo-characters $(\gamma_i,\gamma_i')$, $i=1,2$,
\begin{align*}
\Tri_{M_{n_1,n_2}^{\star}}(\tau_1\otimes\tau_2,\theta_{n_1,1,\gamma_1}\otimes(\theta_{n_2,1,\gamma_2})^{\square},\theta_{n_1,1,\gamma_1'}\otimes(\theta_{n_2,1,\gamma_2'})^{\square})\ne0.
\end{align*}

Let $\sigma$ (resp. $\sigma'$) be an irreducible summand of $\theta_{n_2,1,\gamma_2}^{\square}$
(resp. $\theta_{n_2,1,\gamma_2'}^{\square}$), such that
\begin{align}\label{space:heredity proof}
\Tri_{M_{n_1,n_2}^{\star}}(\tau_1\otimes\tau_2,\theta_{n_1,1,\gamma_1}\otimes\sigma,\theta_{n_1,1,\gamma_1'}\otimes\sigma')\ne0.
\end{align}
If $n_2$ is odd $\sigma=\theta_{n_2,1,\gamma_2}^{\square}$ and $\sigma'=\theta_{n_2,1,\gamma_2'}^{\square}$ (\cite{Kable} Proposition~3.1).

According to \eqref{eq:refinement for tensor}, there exists a pseudo-character $\gamma$ with
\begin{align}\label{eq:heredity tensor and induced}
&\theta_{n_1,1,\gamma_1}\widetilde{\otimes}_{\gamma}\theta_{n_2,1,\gamma_2}=\ind_{\widetilde{M}_{n_1,n_2}^{\star}}^{\widetilde{M}_{n_1,n_2}}(\theta_{n_1,1,\gamma_1}\otimes \sigma).
\end{align}
Indeed, this is clear if $n_1$ and $n_2$ have the same parity (e.g., if both are even, any summand $\sigma$ is suitable for any $\gamma$).
Regarding the last case of odd $n_1$ and even $n_2$, first note that $\sigma$ admits a character on $\widetilde{Z}_{n_2}$
($\widetilde{Z}_{n_2}$ is abelian and
central in $\widetilde{\GL}_{n_2}^{\square}$), which is trivial on $\mathfrak{s}(Z_{n_2}^2)$ (because $\omega_{\theta_{n_2,1,\gamma_2}}(\mathfrak{s}(Z_{n_2}^2))=1$).
Hence the mapping $\gamma_3(z)=\sigma(\mathfrak{s}(z))$ is a pseudo-character. Therefore, in this case $\gamma$ is determined by the condition
$\gamma=\gamma_1\gamma_3$ (this is a pseudo-character because here $\lfloor n_1/2\rfloor+\lfloor n_2/2\rfloor=\lfloor n/2\rfloor$).

Similarly, for some $\gamma'$,
\begin{align*}
&\theta_{n_1,1,\gamma_1'}\widetilde{\otimes}_{\gamma'}\theta_{n_2,1,\gamma_2'}=\ind_{\widetilde{M}_{n_1,n_2}^{\star}}^{\widetilde{M}_{n_1,n_2}}(\theta_{n_1,1,\gamma_1'}\otimes \sigma').
\end{align*}

Applying Frobenius reciprocity to \eqref{eq:result of Kable for Jacquet module of exceptional} and using \eqref{eq:heredity tensor and induced}, we see that $\theta_{n,1,\gamma}$ is a subrepresentation of
\begin{align}\label{eq:tensor in heredity proof 0}
\Ind_{\widetilde{Q}_{n_1,n_2}}^{\widetilde{\GL}_n}(\delta_{Q_{n_1,n_2}}^{-1/4}\ind_{\widetilde{M}_{n_1,n_2}^{\star}}^{\widetilde{M}_{n_1,n_2}}(\theta_{n_1,1,\gamma_1}\otimes \sigma)),
\end{align}
which is isomorphic to
\begin{align}\label{eq:tensor in heredity proof}
\ind_{\widetilde{Q}_{n_1,n_2}^{\star}}^{\widetilde{\GL}_n}(\delta_{Q_{n_1,n_2}}^{1/4}(\theta_{n_1,1,\gamma_1}\otimes\sigma)).
\end{align}
Note that the induction from $\widetilde{Q}_{n_1,n_2}^{\star}$ is not normalized. The isomorphism is given by
$\varphi_0\mapsto\varphi$ where $\varphi(g)=\varphi_0(g)(1)$ ($\varphi_0$ in \eqref{eq:tensor in heredity proof 0}).

Regard an element $\varphi$ in the space of $\theta_{n,1,\gamma}$ as a function in \eqref{eq:tensor in heredity proof}.
Specifically, $\varphi$ is a function on $\widetilde{\GL}_n$, compactly supported modulo $\widetilde{Q}_{n_1,n_2}^{\star}$, taking values in the space of $\theta_{n_1,1,\gamma_1}\otimes \sigma$, and satisfying for
$m\in\widetilde{M}_{n_1,n_2}^{\star}$, $u\in U_{n_1,n_2}$ and $g\in\widetilde{\GL}_n$,
\begin{align*}
\varphi(mug)=\delta_{Q_{n_1,n_2}}^{1/4}(m)(\theta_{n_1,1,\gamma_1}\otimes\sigma)(m)\varphi(g).
\end{align*}
Similar properties holds for $\varphi'$ in the space of $\theta_{n,1,\gamma'}$.

Let $L\ne0$ belong to \eqref{space:heredity proof} and take $f$ in the space of $\tau$. Then
for $q\in Q_{n_1,n_2}^{\star}$,
\begin{align*}
L(f(q),\varphi(q),\varphi'(q))=\delta_{Q_{n_1,n_2}}(q)L(f(1),\varphi(1),\varphi'(1)).
\end{align*}
Hence the following integral is (formally) well defined (see e.g. \cite{BZ1} 1.21),
\begin{align*}
T(f,\varphi,\varphi')=\int_{Q_{n_1,n_2}^{\star}\bsl\GL_n}L(f(g),\varphi(g),\varphi'(g))\ dg.
\end{align*}
It is absolutely convergent because according to the Iwasawa decomposition, it is equal to
\begin{align}\label{int:Iwasawa}
\int_{K}\int_{Q_{n_1,n_2}^{\star}\bsl Q_{n_1,n_2}}L(f(qk),\varphi(qk),\varphi'(qk))\delta_{Q_{n_1,n_2}}^{-1}(q)\ d_rq\ dk,
\end{align}
where $K$ is the hyperspecial compact open subgroup of $\GL_n$.

Since $T\in\Tri_{\GL_n}(\tau,\theta_{n,1,\gamma},\theta_{n,1,\gamma'})$, it is left to show $T\ne0$. Assume $L(x,y,y')\ne0$ for corresponding data $x,y$ and $y'$.
Take $f$ supported on $Q_{n_1,n_2}\mathcal{V}$, where $\mathcal{V}$ is a small compact open neighborhood of the identity in $\GL_{n}$,
and
\begin{align}
f(muv)=\delta_{Q_{n_1,n_2}}^{1/2}(m)(\tau_1\otimes\tau_2)(m)x,\qquad\forall m\in M_{n_1,n_2}, u\in U_{n_1,n_2}, v\in\mathcal{V}.
\end{align}

Next we show that there is $\varphi$ in the space of $\theta_{n,1,\gamma}$ with
\begin{align}\label{eq:heredity requirement from varphi}
\varphi(m)=\begin{dcases}\delta_{Q_{n_1,n_2}}^{1/4}(m)(\theta_{n_1,1,\gamma_1}\otimes\sigma)(m)y&m\in \widetilde{M}_{n_1,n_2}^{\star},\\
0&m\in\widetilde{M}_{n_1,n_2}-\widetilde{M}_{n_1,n_2}^{\star}.
\end{dcases}
\end{align}
Indeed, take $\varphi_0$ in the space of $\theta_{n,1,\gamma}$  as a subrepresentation of \eqref{eq:tensor in heredity proof 0},
such that $\varphi_0(1)\ne0$. Let $\beta_y$ be the function in the space of $\ind_{\widetilde{M}_{n_1,n_2}^{\star}}^{\widetilde{M}_{n_1,n_2}}(\theta_{n_1,1,\gamma_1}\otimes\sigma)$, which vanishes
on $\widetilde{M}_{n_1,n_2}-\widetilde{M}_{n_1,n_2}^{\star}$ and for $m\in\widetilde{M}_{n_1,n_2}^{\star}$,
$\beta_y(m)=(\theta_{n_1,1,\gamma_1}\otimes\sigma)(m)y$. Because $\theta_{n_1,1,\gamma_1}$ and $\theta_{n_2,1,\gamma_2}$ are irreducible, so is their metaplectic tensor,
whence by \eqref{eq:heredity tensor and induced}, $\ind_{\widetilde{M}_{n_1,n_2}^{\star}}^{\widetilde{M}_{n_1,n_2}}(\theta_{n_1,1,\gamma_1}\otimes\sigma)$ is
irreducible. Hence there is $m_0\in M_{n_1,n_2}$ such that $\varphi_0(m_0)=\beta_y$.
Let $\varphi$ be the image in \eqref{eq:tensor in heredity proof} of $m_0\varphi_0$. Then $\varphi$ satisfies \eqref{eq:heredity requirement from varphi}.

The same argument applies to $\theta_{n,1,\gamma'}$ and denote the corresponding function by $\varphi'$, it satisfies
\eqref{eq:heredity requirement from varphi} with respect to $y'$ (and $\theta_{n_1,1,\gamma_1'}\otimes\sigma'$).

Using \eqref{int:Iwasawa} and $Q_{n_1,n_2}\mathcal{V}\cap K=(Q_{n_1,n_2}\cap K)\mathcal{V}$,
\begin{align*}
T(f,\varphi,\varphi')=
\int_{(Q_{n_1,n_2}\cap K)\mathcal{V}}\int_{Q_{n_1,n_2}^{\star}\bsl Q_{n_1,n_2}}L(f(qk),\varphi(qk),\varphi'(qk))\delta_{Q_{n_1,n_2}}^{-1}(q)\ d_rq\ dk.
\end{align*}
For a sufficiently small $\mathcal{V}$
(with respect to $\varphi$ and $\varphi'$), the $dk$-integration may be ignored. Using the fact that
$Q_{n_1,n_2}^{\star}\bsl Q_{n_1,n_2}=M_{n_1,n_2}^{\star}\bsl M_{n_1,n_2}$ is finite, we obtain (up to a nonzero measure constant)
\begin{align*}
\sum_{m\in M_{n_1,n_2}^{\star}\bsl M_{n_1,n_2}}L(f(m),\varphi(m),\varphi'(m))\delta_{Q_{n_1,n_2}}^{-1}(m).
\end{align*}
 According to \eqref{eq:heredity requirement from varphi} this equals $L(x,y,y')$ which is nonzero.
We conclude that $\tau$ is distinguished. 
\end{proof}

\subsection{Combinatorial characterization}\label{subsection:only if part}
We characterize distinguished principal series representations in terms of their inducing data. We start with the case of $\GL_2$, which in fact can be reproduced from the results of Savin \cite{Savin3} for $\GL_3$. The argument is provided, for completeness and because this case will be used as a base case in the course of proving Theorem~\ref{theorem:distinguished principal series}.
\begin{claim}\label{claim:distinguished principal series n=2}
Let $\tau$ be a principal series representation of $\GL_2$, induced from the character
$\eta_1\otimes\eta_2$ of $T_2$. Then
$\tau$ is distinguished if and only if $\eta_1^2=\eta_2^2=1$ or $\eta_2=\eta_1^{-1}$.
\end{claim}
\begin{proof}[Proof of Claim~\ref{claim:distinguished principal series n=2}]
Put $\theta=\theta_{2,1,\gamma_{\psi}}$ and $\theta'=\theta_{2,1,\gamma_{\psi^{-1}}}$.
By the Frobenius reciprocity,
\begin{align}\label{eq:distinguished characterization n=2}
\Hom_{\GL_2}(\theta\otimes\theta',\tau^{\vee})=
\Hom_{T_2}((\theta\otimes\theta')_{N_2},\delta_{B_2}^{1/2}(\eta_1^{-1}\otimes\eta_2^{-1})).
\end{align}
Assume that $\eta_1\otimes\eta_2$ is of the prescribed form, we prove that $\tau$ is distinguished. If $\eta_1^2=\eta_2^2=1$, then $\eta_1$ and $\eta_2$ are distinguished
and the result follows from Theorem~\ref{theorem:upper heredity of dist}. Now assume $\eta_2=\eta_1^{-1}$ and $\eta_1^2\ne1$.

We use the geometric realization of $\theta$ given in \cite{Savin3,FKS,Flicker}. Let $\mathcal{S}(F)$ be the space of Schwartz-Bruhat functions on $F$. Let $C_2\subset\mathcal{S}(F)$ be the subspace of functions
$f$, for which there is a constant $A_f>0$ satisfying $f(x)=0$ for $|x|>A_f$ and $f(y^2x)=f(x)$ for all $x,y$ with $|y^2x|,|x|<A_f^{-1}$. Also for $m\in\QQ$, let $C_{2}^m$ denote the space
of functions $f$ for which $|\cdot|^{m}f\in C_2$. The representation $\theta$ can be realized on $C_2^{1/4}$. In particular, the action of $N_2$ is given by $\left(\begin{smallmatrix}1&u\\&1\end{smallmatrix}\right)\cdot f(x)=\psi(u x)f(x)$ (\cite{Savin3} p.~372).

Following the arguments of Savin \cite{Savin3} (proof of Proposition~6), one sees that $(\theta\otimes\theta')_{N_2}$ is embedded in $C_2^{1/2}$, and under
this embedding $(\theta(N_2)\otimes\theta'(N_2))_{N_2}=\mathcal{S}(F^*)$. If $f\in C_2^{1/2}$ is
the image of an element from $(\theta\otimes\theta')_{N_2}$, the action of $T_2$ is given by
\begin{align*}
diag(a,b)\cdot f(x)=|ab^{-1}|f(xab^{-1})
\end{align*}
(see \cite{Savin3} p.~372 and use $\gamma_{\psi^{-1}}=\gamma_{\psi}^{-1}$).
According to \cite{Savin3} (Proposition~4), for a character $\mu$ of $F^*$ such that $\mu^2\ne1$, the nontrivial functional on $\mathcal{S}(F^*)$ given by
 \begin{align*}
 f\mapsto\int_{F^*}f(x)\mu(x^{-1})d^*x
 \end{align*}
 extends to $C_2$. Since for $f\in C_2^{1/2}$, $|\cdot|^{1/2}f\in C_2$, and $\eta_1^2\ne1$, it follows that
 \begin{align*}
 f\mapsto\int_{F^*}|x|^{1/2}f(x)\eta_1^{-1}(x^{-1})d^*x
 \end{align*}
 defines a functional in
\begin{align*}
\Hom_{T_2}((\theta\otimes\theta')_{N_2},\delta_{B_2}^{1/2}(\eta_1^{-1}\otimes\eta_1)),
\end{align*}
 which is nontrivial because it does not vanish on the subspace $(\theta(N_2)\otimes\theta'(N_2))_{N_2}$. Looking at \eqref{eq:distinguished characterization n=2} we
 see that $\tau$ is distinguished.

In the other direction assume that $\tau$ is distinguished. By virtue of Lemma~\ref{lemma:composition series of Jacquet of tensor},
either
\begin{align*}
\Hom_{T_2}(\theta_{N_2}\otimes\theta'_{N_2},\delta_{B_2}^{1/2}(\eta_1^{-1}\otimes\eta_2^{-1}))\ne0
\end{align*}
or
\begin{align*}
& \Hom_{T_2}((\theta(N_2)\otimes\theta'(N_2))_{N_2},\delta_{B_2}^{1/2}(\eta_1^{-1}\otimes\eta_2^{-1}))\ne0.
\end{align*}

In the former case, by \eqref{eq:result of Kable for Jacquet module of exceptional} we have
\begin{align*}
\Hom_{T_2}((\theta_{1,1,1}\widetilde{\otimes}_{\gamma_{\psi}}\theta_{1,1,1})\otimes
(\theta_{1,1,1}\widetilde{\otimes}_{\gamma_{\psi^{-1}}}\theta_{1,1,1}),\eta_1^{-1}\otimes\eta_2^{-1})\ne0.
\end{align*}
In particular when restricting to $T_2^2$ we obtain (see Section~\ref{subsection:The metaplectic tensor})
\begin{align*}
\Hom_{T_2^2}((\theta_{1,1,1}^{\square}\otimes\theta_{1,1,1}^{\square})\otimes
(\theta_{1,1,1}^{\square}\otimes\theta_{1,1,1}^{\square}),\eta_1^{-1}\otimes\eta_2^{-1})\ne0,
\end{align*}
whence $\eta_1^2=\eta_2^2=1$.

In the latter case
\begin{align*}
\Hom_{T_2}(\mathcal{S}(F^*),\delta_{B_2}^{1/2}(\eta_1^{-1}\otimes\eta_2^{-1}))\ne0.
 \end{align*}
 The action of $T_2$ on $\mathcal{S}(F^*)$ was given above and immediately implies $\eta_1^{-1}=\eta_2$.
 \end{proof}
\begin{remark}\label{remark:Kable's result using analytic methods after n=2}
If $\eta_1$ is unramified and $|2|=1$, the fact that $\Ind_{B_2}^{\GL_2}(\eta_1\otimes\eta_1^{-1})$ is distinguished follows as a particular case from a result of Kable (\cite{Kable2} Theorem~5.4). 
\end{remark}

Let $\eta=\eta_1\otimes\ldots\otimes\eta_n$ be a character of $T_n$. We say that $\eta$ satisfies
condition $(\star)$ if, up to a permutation of the characters $\eta_i$, there is $0\leq k\leq\lfloor n/2\rfloor$ such that
\begin{itemize}[leftmargin=*]
\item $\eta_{2i}=\eta_{2i-1}^{-1}$ for $1\leq i\leq k$,
\item $\eta_i^2=1$ (i.e., $\eta_i$ is distinguished) for $2k+1\leq i\leq n$.
\end{itemize}
Claims~\ref{claim:distinguished for n=1} and \ref{claim:distinguished principal series n=2} state that for $n=1,2$,
the principal series representation induced from $\eta$ is distinguished if and only if $\eta$ satisfies $(\star)$.
The main goal of this work is to extend this to irreducible principal series representations of $\GL_n$, for arbitrary $n$.
Henceforth we proceed under the mild assumption that for $n>3$, an exceptional representation does not have a Whittaker model
(see Section~\ref{subsection:The exceptional representations}).

\begin{remark}\label{remark:usage of |2|=1}
This assumption 
is only needed in Lemma~\ref{lemma:only if direction, odd n or even with restriction} below, used by Proposition~\ref{proposition:only if direction, for any principal series}: in that lemma we use the fact that the second derivative of
an exceptional representation is its highest nonzero derivative (\cite{Kable} Theorems~5.3 and 5.4), this is true if and only if this representation has no Whittaker model.
\end{remark}

\begin{proposition}\label{proposition:only if direction, for any principal series}
Let $\eta$ be a character of $T_n$.
Assume that for some $c\in\{0,1\}^{n-1}$,
\begin{align}\label{eq:assumption main proposition for only if characterization}
\Hom_{T_n}((\mathcal{L}^{n,1}_{c}\theta_{n,1,\gamma}\otimes\mathcal{L}^{n,1}_{c}\theta_{n,1,\gamma'})_{N_n},\delta_{B_n}^{1/2}\eta)\ne0.
\end{align}
Then $\eta$ satisfies $(\star)$.
\end{proposition}
The proof is given below.
Now we prove Theorem~\ref{theorem:distinguished principal series}. Namely, let $\tau$ be a principal series representation of
$\GL_n$ induced from $\eta$. If $\tau$ is distinguished, $\eta$ satisfies $(\star)$. Conversely, if $\tau$ is irreducible and
$\eta$ satisfies $(\star)$, then $\tau$ is distinguished.
\begin{proof}[Proof of Theorem~\ref{theorem:distinguished principal series}]
Assume that $\tau$ is distinguished. According to the Frobenius reciprocity, for some $\gamma$ and $\gamma'$,
\begin{align}\label{eq:Frobenius characterization principal series proof}
\Hom_{T_n}((\theta_{n,1,\gamma}\otimes\theta_{n,1,\gamma'})_{N_n},\delta_{B_n}^{1/2}\eta^{-1})\ne0.
\end{align}
Lemma~\ref{lemma:composition series of Jacquet of tensor} implies that the condition of
Proposition~\ref{proposition:only if direction, for any principal series} holds and the result follows from the proposition.

In the other direction, if $\eta$ satisfies $(\star)$, then since $\tau$ is irreducible,
permuting the inducing data does not change $\tau$. Hence we may assume that $\tau$ is induced from
\begin{align*}
(\eta_1\otimes\eta_1^{-1})\otimes\ldots\otimes(\eta_k\otimes\eta_k^{-1})\otimes\eta_{2k+1}\otimes\ldots\otimes\eta_{n},
\end{align*}
where $0\leq k\leq \lfloor n/2\rfloor$ and for $i>2k$, $\eta_i^2=1$. Now
Claims~\ref{claim:distinguished for n=1} and \ref{claim:distinguished principal series n=2} and
Theorem~\ref{theorem:upper heredity of dist} imply
that $\tau$ is distinguished.
\end{proof}

Theorem~\ref{theorem:Savin's conjecture} - the characterization of distinguished spherical representations of $\GL_n$,
is easily seen to follow from Theorem~\ref{theorem:distinguished principal series}
and the description of the tautologial lift (see e.g. \cite{Kable2} Section~6).
\begin{proof}[Proof of Theorem~\ref{theorem:Savin's conjecture}]
The ``if" part was proved by Kable \cite{Kable2}. If $\tau$ is the irreducible unramified quotient of $\Ind_{B_n}^{\GL_n}(\eta)$ and $\tau$ is distinguished, for some $\gamma$ and $\gamma'$,
\begin{align*}
\Hom_{\GL_n}(\theta_{n,1,\gamma}\otimes\theta_{n,1,\gamma'},\Ind_{B_n}^{\GL_n}(\eta^{-1}))\ne0.
\end{align*}
Hence by Theorem~\ref{theorem:distinguished principal series} the character $\eta$ satisfies $(\star)$. Since there are exactly two unramified square trivial characters of $F^*$, we may assume, perhaps after applying a permutation, that
$\eta$ takes the form
\begin{align*}
\eta_1\otimes\ldots\otimes\eta_{\lfloor n/2\rfloor}\otimes\eta_0\otimes\eta_{\lfloor n/2\rfloor}^{-1}\otimes\ldots\otimes\eta_1^{-1}.
\end{align*}
Here the character $\eta_0$ appears only when $n$ is odd. Because we assumed $\omega_{\tau}=1$,
$\eta_0=1$. This implies that $\tau$ is the lift of a representation of $SO_{2\lfloor n/2\rfloor}$ in the even case, or
$Sp_{2\lfloor n/2\rfloor}$ when $n$ is odd (\cite{Kable2} Section~6).
\end{proof}

\begin{proof}[Proof of Proposition~\ref{proposition:only if direction, for any principal series}]
The idea is to reduce the proof to a computation on $\widetilde{Y}_n$ (or $Y_n$)-modules. This is possible if $c_1=0$. Then we appeal to the results of Section~\ref{section:geometric}. Specifically, the $\widetilde{B}_n^{\circ}$-modules
$\mathcal{L}^{n,1}_c\theta_{n,1,\gamma}$ and $\mathcal{L}^{n,1}_c\theta_{n,1,\gamma'}$ are described using Lemma~\ref{lemma:L of Y_n module}; the Jacquet module of their tensor, which is a representation of $B_n^{\circ}$, is analyzed using Lemma~\ref{lemma:Jacquet functor without character}. Then $\eta$ becomes a quotient of a representation induced from $T_{n-2}$ to $T_n^{\circ}$. The passage from $T_n^{\circ}$ to $T_n$ depends on the parity of $n$.

To simplify the notation, denote $\theta=\theta_{n,1,\gamma}$, or $\theta_{n}$ to clarify the dimension. Similarly,
$\theta'=\theta_{n,1,\gamma'}$. The actual pseudo-characters $\gamma$ and $\gamma'$ may vary during the proof, but this will not bare any impact on the arguments.

Our first step is to reduce the proof to the case $c_1=0$. To this end we claim the following.
\begin{lemma}\label{lemma:reducing n by 1 when b starts with 1}
Let $b\in\{0,1\}^{n-2}$, $n\geq3$. If
\begin{align*}
\Hom_{T_n}((\mathcal{L}^{n,1}_{(1,b)}\theta\otimes\mathcal{L}^{n,1}_{(1,b)}\theta')_{N_n},\delta_{B_n}^{1/2}\eta_1\otimes\ldots\otimes\eta_n)\ne0,
\end{align*}
then $\eta_n^{2}=1$ and for some character $\epsilon$ of $F^*$ such that $\epsilon^2=1$,
\begin{align*}
\Hom_{T_{n-1}}((\mathcal{L}^{n-1,1}_{b}\theta_{n-1}\otimes\mathcal{L}^{n-1,1}_{b}\theta'_{n-1})_{N_{n-1}},\delta_{B_{n-1}}^{1/2}
\epsilon\eta_1\otimes\ldots\otimes\epsilon\eta_{n-1})\ne0.
\end{align*}
\end{lemma}
\begin{remark}\label{remark:eta twisted by epsilon preserves star}
In general $\eta_1\otimes\ldots\otimes\eta_n$ satisfies $(\star)$ if and only if
$\epsilon\eta_1\otimes\ldots\otimes\epsilon\eta_n$ does, because $\epsilon^2=1$.
\end{remark}

Next we state the following lemma.
\begin{lemma}\label{lemma:only if direction, odd n or even with restriction}
Let $n\geq3$ and assume Proposition~\ref{proposition:only if direction, for any principal series} holds for
all $n_0\leq n-2$. Let $\eta$ be a character of $T_n$ such that \eqref{eq:assumption main proposition for only if characterization}
holds with $c_1=0$. If $n$ is even, assume $\eta_n^2=1$. Then $\eta$ satisfies $(\star)$.
\end{lemma}
Before proving the lemmas, let us derive the proposition. According to
Claim~\ref{claim:distinguished for n=1} and the proof of Claim~\ref{claim:distinguished principal series n=2},
Proposition~\ref{proposition:only if direction, for any principal series} holds for $n=1,2$.

Let $n=3$. If $c_1=0$, $\eta$ satisfies $(\star)$ according to Lemma~\ref{lemma:only if direction, odd n or even with restriction}. Otherwise $c_1=1$, by Lemma~\ref{lemma:reducing n by 1 when b starts with 1} we get $\eta_3^2=1$ and
reduce the problem to $n=2$, and then either $\eta_1^2=\eta_2^2$ or $\eta_1^{-1}=\eta_2$. This completes the proof (in fact, for $n=3$ the proof can also be deduced from \cite{Savin}).

Let $n\geq 3$ be odd and assume
Proposition~\ref{proposition:only if direction, for any principal series} holds for $n_0\leq n$.
We deduce it for $n+1$.

\begin{enumerate}[leftmargin=*]
\item $c_1=1$\label{item:c_1=1}: Apply Lemma~\ref{lemma:reducing n by 1 when b starts with 1} to $\eta_1\otimes\ldots\otimes\eta_{n+1}$
and obtain $\eta_{n+1}^2=1$ and
\begin{align*}
\Hom_{T_{n}}((\mathcal{L}^{n,1}_{b}\theta_{n}\otimes\mathcal{L}^{n,1}_{b}\theta'_{n})_{N_n},\delta_{B_{n}}^{1/2}
\epsilon\eta_1\otimes\ldots\otimes\epsilon\eta_{n})\ne0.
\end{align*}
Then we can apply Proposition~\ref{proposition:only if direction, for any principal series} with $n$ and deduce that
$\eta_1\otimes\ldots\otimes\eta_{n}$ satisfies $(\star)$ and so does $\eta_1\otimes\ldots\otimes\eta_{n+1}$.

\item $c_1=0$ and $\eta_{n+1}^2=1$: the result follows from Lemma~\ref{lemma:only if direction, odd n or even with restriction} ($n+1$ is even).

\item $c_1=0$ and $\eta_{n+1}^2\ne1$. Let $\eta_0$ be a character of $F^*$ such that $\eta_0^2=1$. According to
Theorem~\ref{theorem:upper heredity of dist}, the principal series representation induced from the character $\eta_0\otimes\eta_1\otimes\ldots\otimes\eta_{n+1}$ of $T_{n+2}$ is
distinguished. Hence for some $d\in\{0,1\}^{n+1}$,
\begin{align*}
\Hom_{T_{n+2}}((\mathcal{L}^{n+2,1}_{d}\theta_{n+2}\otimes\mathcal{L}^{n+2,1}_{d}\theta'_{n+2})_{N_{n+2}},\delta_{B_{n+2}}^{1/2}
\eta_0\otimes\ldots\otimes\eta_{n+1})\ne0.
\end{align*}
Since $\eta_{n+1}^2\ne1$, Lemma~\ref{lemma:reducing n by 1 when b starts with 1} implies $d_1=0$. Therefore we may apply
Lemma~\ref{lemma:only if direction, odd n or even with restriction} and deduce that $\eta_0\otimes\ldots\otimes\eta_{n+1}$ satisfies
$(\star)$ and so does $\eta_1\otimes\ldots\otimes\eta_{n+1}$.
\end{enumerate}

Having established Proposition~\ref{proposition:only if direction, for any principal series} for $n+1$ unconditionally, we
handle $n+2$. If $c_1=1$, we may proceed as in \eqref{item:c_1=1} and apply Lemma~\ref{lemma:reducing n by 1 when b starts with 1}. We reduce the
proof to a character of $T_{n+1}$ where the result is now known to hold. Finally if $c_1=0$ we appeal directly to
Lemma~\ref{lemma:only if direction, odd n or even with restriction}.
\end{proof}

\begin{proof}[Proof of Lemma~\ref{lemma:reducing n by 1 when b starts with 1}]
Since $U_{n-1,1}$ acts trivially on $\theta_{U_{n-1,1}}$ and $U_{n-l,l}\triangleleft N_n$ for all $l$,
$\mathcal{L}^{n,1}_{(1,b)}\theta=\mathcal{L}^{n-1,1}_{b}\theta_{U_{n-1,1}}$.
Hence the assumption implies
\begin{align}\label{eq:reducing n by 1 when b starts with 1 first observation}
\Hom_{T_n}((\mathcal{L}^{n-1,1}_{b}\theta_{U_{n-1,1}}\otimes\mathcal{L}^{n-1,1}_{b}\theta'_{U_{n-1,1}})_{N_{n-1}},\delta_{B_{n}}^{1/2}\eta_1\otimes\ldots\otimes\eta_n)\ne0.
\end{align}
Put $T_n^{\star}=\{diag(t_1,\ldots,t_{n-1},t_n^2):t_i\in F^*\}$. Then
\begin{align*}
(\mathcal{L}^{n-1,1}_{b}\theta_{U_{n-1,1}})|_{p^{-1}(T_n^{\star}N_{n-1})}=
(\mathcal{L}^{n-1,1}_{b}(\theta_{U_{n-1,1}}|_{p^{-1}(\GL_{n-1}\times\GL_{1}^{\square})}))|_{p^{-1}(T_n^{\star}N_{n-1})}.
\end{align*}
According to \eqref{eq:result of Kable for Jacquet module of exceptional} and
Claim~\ref{claim:Mackey theory applied to metaplectic tensor},
\begin{align*}
&\delta_{Q_{n-1,1}}^{-1/4}\theta_{U_{n-1,1}}|_{p^{-1}(\GL_{n-1}\times\GL_{1}^{\square})}=\xi\otimes\theta_1^{\square},\\
&\xi=\begin{dcases}
\bigoplus_{g\in F^{*2}\bsl F^*}\chi_g\theta_{n-1}&\text{$n-1$ is odd,}\\
\theta_{n-1}&\text{$n-1$ is even.}
\end{dcases}
\end{align*}
Here $\chi_g(x)=(\det{x},g)$. The tensor $\xi\otimes\theta_1^{\square}$ commutes with the application of $\mathcal{L}^{n-1,1}_{b}$ whence
\begin{align*}
\mathcal{L}^{n-1,1}_{b}(\theta_{U_{n-1,1}}|_{p^{-1}(\GL_{n-1}\times\GL_{1}^{\square})})
=\delta_{Q_{n-1,1}}^{1/4}(\mathcal{L}^{n-1,1}_{b}\xi)\otimes\theta_1^{\square}.
\end{align*}
Therefore
\begin{align*}
(\mathcal{L}^{n-1,1}_{b}\theta_{U_{n-1,1}}\otimes
\mathcal{L}^{n-1,1}_{b}\theta'_{U_{n-1,1}})|_{T_n^{\star}N_{n-1}}
=\delta_{Q_{n-1,1}}^{1/2}
(\mathcal{L}^{n-1,1}_{b}\xi\otimes\mathcal{L}^{n-1,1}_{b}\xi')|_{B_{n-1}}\otimes(\theta_1^{\square}\otimes{\theta'}_1^{\square})|_{F^{*2}},
\end{align*}
where $\xi'$ is defined as $\xi$, with respect to $\theta'_{n-1}$.
Also $\delta_{Q_{n-1,1}}^{-1/2}\delta_{B_n}^{1/2}=\delta_{B_{n-1}}^{1/2}$. Plugging these observations
into \eqref{eq:reducing n by 1 when b starts with 1 first observation} yields
\begin{align*}
\Hom_{T_n^{\star}}(
(\mathcal{L}^{n-1,1}_{b}\xi\otimes\mathcal{L}^{n-1,1}_{b}\xi')_{N_{n-1}}\otimes(\theta_1^{\square}\otimes{\theta'}_1^{\square}),
\delta_{B_{n-1}}^{1/2}\eta_1\otimes\ldots\otimes\eta_n)\ne0.
\end{align*}
This implies $\eta_n^2=1$ and
\begin{align*}
\Hom_{T_{n-1}}(
(\mathcal{L}^{n-1,1}_{b}\xi\otimes\mathcal{L}^{n-1,1}_{b}\xi')_{N_{n-1}},
\delta_{B_{n-1}}^{1/2}\eta_1\otimes\ldots\otimes\eta_{n-1})\ne0.
\end{align*}

Now if $n-1$ is even, the second assertion follows immediately. Assume $n-1$ is odd. Because $\mathcal{L}^{n-1,1}_{b}$ commutes with finite
direct sums and $\chi_g$ is a character of $\GL_{n-1}$, we reach
\begin{align*}
\Hom_{T_{n-1}}(
\bigoplus_{g,g'\in F^{*2}\bsl F^*}
\chi_{g}\chi_{g'}(
 \mathcal{L}^{n-1,1}_{b}\theta_{n-1}\otimes\mathcal{L}^{n-1,1}_{b}\theta'_{n-1})_{N_{n-1}},
\delta_{B_{n-1}}^{1/2}\eta_1\otimes\ldots\otimes\eta_{n-1})\ne0.
\end{align*}
As a character of $T_{n-1}$, $\chi_g(t)\chi_{g'}(t)=\chi_{gg'}(t)=\prod_{i=1}^{n-1}\epsilon_{gg'}(t_i)$,
where $\epsilon_{gg'}(x)=(x,gg')$ ($x\in F^*$). Hence for some $h\in F^*$,
\begin{align*}
\Hom_{T_{n-1}}((\mathcal{L}^{n-1,1}_{b}\theta_{n-1}\otimes\mathcal{L}^{n-1,1}_{b}\theta'_{n-1})_{N_{n-1}},
\delta_{B_{n-1}}^{1/2}\epsilon_h\eta_1\otimes\ldots\otimes\epsilon_h\eta_{n-1})\ne0.
\end{align*}
Clearly $\epsilon_h^2=1$ and the claim is proved.
\end{proof}

\begin{proof}[Proof of Lemma~\ref{lemma:only if direction, odd n or even with restriction}]
Fix $c\in\{0,1\}^{n-1}$ with $c_1=0$.
Kable \cite{Kable} (Theorem~5.3) 
showed that the second derivative of $\theta=\theta_{n,1,\gamma}$ is $|\det|^{-1/2}\theta_{n-2,1,\gamma_{\psi}^{-1}\gamma}$, where $\psi$
is the character with respect to which the derivative is defined,
and furthermore, this is the highest nonzero derivative (\cite{Kable} Theorem~5.4). Hence the kernel of the Jacquet functor
$\theta|_{\widetilde{Y}_n}(U_{n-1,1})$ is equal, as a $\widetilde{Y}_n$-module, to the application of $\Phi^+\Psi^+$ to the second derivative (\cite{BZ2} 3.2 (e) and 3.5). Put $\theta_{n-2}=\theta_{n-2,1,\gamma_{\psi}^{-1}\gamma}$.
Thus
\begin{align*}
\theta|_{\widetilde{Y}_n}(U_{n-1,1})=\ind_{\widetilde{Y}_{n-1}U_{n-1,1}}^{\widetilde{Y}_n}(\delta_{Y_{n-1}}^{1/2}
\ind_{\widetilde{\GL}_{n-2}U_{n-2,1}}^{\widetilde{Y}_{n-1}(F)}(|\det|^{1/2}|\det|^{-1/2}\theta_{n-2})\otimes\psi).
\end{align*}
Since $\delta_{Y_{n-1}}|_{\GL_{n-2}}=|\det|$, in the notation of Section~\ref{section:geometric},
\begin{align*}
\theta|_{\widetilde{Y}_n}(U_{n-1,1})=\Phi_{\diamond}^+\Psi_{\diamond}^+(\pi_0),\qquad \pi_0=
|\det|^{1/2}\theta_{n-2}.
\end{align*}
Then as $\widetilde{B}_n^{\circ}$-modules
\begin{align}\label{eq:Bn circ modules begin}
\mathcal{L}^{n,1}_{c}(\theta|_{\widetilde{Y}_n})=\mathcal{L}^{n,2}_{b}
\Phi_{\diamond}^+\Psi_{\diamond}^+(\pi_0),\qquad b=(c_2,\ldots,c_{n-1}).
\end{align}
While restriction to $\widetilde{Y}_n$ enables us to apply the results of Section~\ref{section:geometric}, we seem to be
losing the last coordinate of $\widetilde{T}_n$. We explain how to remedy this. Since the pair of subgroups $\widetilde{B}_n^{\circ}$ and $\widetilde{Z}_n^{\mathe}$ are commuting in $\widetilde{\GL}_n$,
one can form the genuine representation $\omega_{\theta}\mathcal{L}^{n,1}_{c}(\theta|_{\widetilde{Y}_n})$ of
$p^{-1}(B_n^{\circ}Z_n^{\mathe})$, which is in fact $\mathcal{L}^{n,1}_{c}(\theta)|_{p^{-1}(B_n^{\circ}Z_n^{\mathe})}$.
If $n$ is odd, $p^{-1}(B_n^{\circ}Z_n^{\mathe})=\widetilde{B}_{n}$ and we recover $\mathcal{L}^{n,1}_{c}(\theta)$.

By virtue of Lemma~\ref{lemma:L of Y_n module} applied to the right-hand side of \eqref{eq:Bn circ modules begin}, if $k$ is defined by the lemma and $b^{k,\vdash}=(b_{k+1},\ldots,b_{n-2})$,
as $\widetilde{B}_n^{\circ}$-modules
\begin{align*}
\mathcal{L}^{n,1}_{c}(\theta|_{\widetilde{Y}_n})=
\mathrm{s.s.}\bigoplus_{m=\min(1,k)}^{k}\mathcal{E}_{n,m+2}
\mathcal{L}^{n-2,m}_{(0^{k-m},b^{k,\vdash})}(\pi_0).
\end{align*}
Then as $p^{-1}(B_n^{\circ}Z_n^{\mathe})$-modules
\begin{align}\label{eq:application of the Yn module lemma and omega theta}
\omega_{\theta}\mathcal{L}^{n,1}_{c}(\theta|_{\widetilde{Y}_n})=
\mathrm{s.s.}\bigoplus_{m=\min(1,k)}^{k}\omega_{\theta}\mathcal{E}_{n,m+2}
\mathcal{L}^{n-2,m}_{(0^{k-m},b^{k,\vdash})}(\pi_0).
\end{align}

Recall that $\theta'=\theta_{n,1,\gamma'}$. We take the derivative with respect to $\psi^{-1}$ (instead of $\psi$). Put $\theta'_{n-2}=\theta_{n-2,1,\gamma_{\psi^{-1}}^{-1}\gamma'}$, $\pi'_0=|\det|^{1/2}\theta'_{n-2}$, $\Phi_{\diamond}^+$ is defined using $\psi^{-1}$ and then
\eqref{eq:application of the Yn module lemma and omega theta} holds for $\theta'$, with $\pi'_0$. Next we see that according to \eqref{eq:action of u in geometric general n},
for any $\min(1,k)\leq m,\widetilde{m}\leq k$,
\begin{align}\label{eq:iso m and m tilde}
&(\mathcal{E}_{n,m+2}
\mathcal{L}^{n-2,m}_{(0^{k-m},b^{k,\vdash})}(\pi_0)\otimes
\mathcal{E}_{n,\widetilde{m}+2}
\mathcal{L}^{n-2,\widetilde{m}}_{(0^{k-\widetilde{m}},b^{k,\vdash})}(\pi'_0))_{U_{n-1,1}}\\\nonumber&=
\begin{cases}
\mathcal{E}_{n,m+2}^-(
\mathcal{L}^{n-2,m}_{(0^{k-m},b^{k,\vdash})}(\pi_0)\otimes
\mathcal{L}^{n-2,m}_{(0^{k-m},b^{k,\vdash})}(\pi'_0))&m=\widetilde{m},\\\nonumber
0&m\ne \widetilde{m}.
\end{cases}
\end{align}
Put
\begin{align*}
\Lambda(b)=\mathcal{L}^{n-2,m}_{(0^{k-m},b^{k,\vdash})}(\pi_0)\otimes
\mathcal{L}^{n-2,m}_{(0^{k-m},b^{k,\vdash})}(\pi'_0).
\end{align*}
This is a representation of $\widetilde{Q}_{1^{n-m-2},m}$ which factors through
$Q_{1^{n-m-2},m}$, therefore we will regard it as a representation of the latter.
\begin{remark}\label{remark:center on Lambda}
If one
takes the derivative of $\theta'$ with respect to $\psi_{\alpha}$, where $\psi_{\alpha}(x)=\psi(\alpha x)$ ($\alpha\in F^*$), instead of $\psi^{-1}$,
the isomorphism \eqref{eq:iso m and m tilde} is deduced using
\begin{align*}
&\ind_{\widetilde{Y}_{n-1}U_{n-1,1}}^{\widetilde{Y}_n}(
\ind_{\widetilde{\GL}_{n-2}U_{n-2,1}}^{\widetilde{Y}_{n-1}(F)}(|\det|^{1/2}\theta_{n-2,1,\gamma_{\psi_{\alpha}}^{-1}\gamma'})\otimes\psi_{\alpha})\\&\isomorphic
\ind_{\widetilde{Y}_{n-1}U_{n-1,1}}^{\widetilde{Y}_n}(
\ind_{\widetilde{\GL}_{n-2}U_{n-2,1}}^{\widetilde{Y}_{n-1}(F)}(|\det|^{1/2}\theta_{n-2,1,\gamma_{\psi^{-1}}^{-1}\gamma'})\otimes\psi^{-1}).
\end{align*}
To see this note that $(-\alpha^{-1},x)\gamma_{\psi_{\alpha}}(x)=\gamma_{\psi^{-1}}(x)$.
\end{remark}
We obtain the following equality of $B_n^{\circ}$-modules
\begin{align*}
(\mathcal{L}^{n,1}_{c}(\theta|_{\widetilde{Y}_n})\otimes
\mathcal{L}^{n,1}_{c}(\theta'|_{\widetilde{Y}_n}))_{U_{n-1,1}}=
\mathrm{s.s.}\bigoplus_{m=\min(1,k)}^{k}\mathcal{E}_{n,m+2}^-(\Lambda(b)).
\end{align*}
Put $i=m+2$, apply Lemma~\ref{lemma:Jacquet functor without character} to each summand and obtain, as $T_n^{\circ}$-modules,
\begin{align*}
&(\mathcal{L}^{n,1}_{c}(\theta|_{\widetilde{Y}_n})\otimes\mathcal{L}^{n,1}_{c}(\theta'|_{\widetilde{Y}_n}))_{N_{n}}=
\mathrm{s.s.}\bigoplus_{m=\min(1,k)}^{k}
\ind_{\rconj{w_i}T_{n-2}}^{T_n^{\circ}}(\delta_{Y_i}^{-1}\ \rconj{w_i}(\Lambda(b)_{N_{n-2}})).
\end{align*}
Here we used $N_{n}=N_{n-1}\ltimes U_{n-1,1}$. This equality yields the following identity
of $T_n^{\circ}Z_n^{\mathe}$-modules,
\begin{align*}
(\omega_{\theta}\mathcal{L}^{n,1}_{c}(\theta|_{\widetilde{Y}_n})\otimes
\omega_{\theta'}\mathcal{L}^{n,1}_{c}(\theta'|_{\widetilde{Y}_n}))_{N_n}=
\mathrm{s.s.}\bigoplus_{m=\min(1,k)}^{k}\omega_{\theta}\omega_{\theta'}\ind_{\rconj{w_i}T_{n-2}}^{T_n^{\circ}}(\delta_{Y_i}^{-1}\ \rconj{w_i}(\Lambda(b)_{N_{n-2}})).
\end{align*}
Thus for some $\min(1,k)\leq m\leq k$,
\begin{align}\label{eq:final hom as torus}
\Hom_{T_n^{\circ}Z_n^{\mathe}}(
\omega_{\theta}\omega_{\theta'}\ind_{\rconj{w_i}T_{n-2}}^{T_n^{\circ}}(\delta_{Y_i}^{-1}\ \rconj{w_i}(\Lambda(b)_{N_{n-2}}))
,\delta_{B_n}^{1/2}\eta_1\otimes\ldots\otimes\eta_n)\ne0.
\end{align}

Next we claim,
\begin{claim}\label{claim:t_x acts trivially}
For $x\in F^{*\mathe}$, the element $t_x=diag(I_{n-i},x,I_{i-2},x)\in T_n^{\circ}Z_n^{\mathe}$ acts on
\begin{align*}
\omega_{\theta}\omega_{\theta'}\ind_{\rconj{w_i}T_{n-2}}^{T_n^{\circ}}(\delta_{Y_i}^{-1}\ \rconj{w_i}(\Lambda(b)_{N_{n-2}}))
\end{align*}
via $\delta_{B_n}^{1/2}(t_x)$.
\end{claim}

The claim and \eqref{eq:final hom as torus} immediately imply $\eta_{n-i+1}(x)\eta_n(x)=1$ for all $x\in F^{*\mathe}$.
Now if $n$ is odd,
we deduce $\eta_{n-i+1}=\eta_n^{-1}$ (!). In the even case we find $\eta_{n-i+1}^2=\eta_n^{-2}$ and according to our assumption
(!!) $\eta_n^2=1$, we obtain $\eta_{n-i+1}^2=1$.

We proceed to apply the induction hypothesis. After restricting \eqref{eq:final hom as torus} to $T_{n}^{\circ}$ and applying Frobenius
reciprocity,
\begin{align}\label{eq:final hom as torus after Frobenius}
\Hom_{\rconj{w_i}T_{n-2}}(\delta_{Y_i}^{-1}\ \rconj{w_i}(\Lambda(b)_{N_{n-2}})
,\delta_{B_n}^{1/2}\eta_1\otimes\ldots\otimes\eta_n)\ne0.
\end{align}
Using
\begin{align*}
&\delta_{B_n}(\rconj{w_i}t)=\prod_{j=1}^{n-i}|t_j|\delta_{B_{n-2}}(t)=|\det{t}|\delta_{Y_i}^{-1}(\rconj{w_i}t)\delta_{B_{n-2}}(t), \qquad \forall t\in T_{n-2},
\end{align*}
we obtain
\begin{align*}
&\Hom_{T_{n-2}}(|\det|^{-1}\Lambda(b)_{N_{n-2}},\delta_{B_{n-2}}^{1/2}\eta^{(n-2)})\ne0,\quad
\eta^{(n-2)}=\eta_1\otimes\ldots\otimes\eta_{n-i}\otimes\eta_{n-i+2}\otimes\ldots\otimes\eta_{n-1}.
\end{align*}
Note that $2\leq i\leq n$ ($i=m+2$). 
Thus
\begin{align*}
\Hom_{T_{n-2}}((\mathcal{L}^{n-2,m}_{(0^{k-m},b^{k,\vdash})}(\theta_{n-2})\otimes
\mathcal{L}^{n-2,m}_{(0^{k-m},b^{k,\vdash})}(\theta'_{n-2}))_{N_{n-2}},\delta_{B_{n-2}}^{1/2}\eta^{(n-2)})\ne0.
\end{align*}

According to Corollary~\ref{corollary:composition series of Jacquet of tensor}, as $B_{n-2}$-modules
\begin{align*}
(\mathcal{L}^{n-2,m}_{(0^{k-m},b^{k,\vdash})}(\theta_{n-2})\otimes
\mathcal{L}^{n-2,m}_{(0^{k-m},b^{k,\vdash})}(\theta'_{n-2}))_{N_{n-2}}=\mathrm{s.s.}\bigoplus_{c^{(n-2)}}
(\mathcal{L}^{n-2,1}_{c^{(n-2)}}(\theta_{n-2})\otimes\mathcal{L}^{n,1}_{c^{(n-2)}}(\theta'_{n-2}))_{N_{n-2}},
\end{align*}
where $c^{(n-2)}$ varies as described in the corollary.
Therefore for some $c^{(n-2)}\in\{0,1\}^{n-3}$,
\begin{align*}
\Hom_{T_{n-2}}((\mathcal{L}^{n-2,1}_{c^{(n-2)}}(\theta_{n-2})\otimes\mathcal{L}^{n,1}_{c^{(n-2)}}(\theta'_{n-2}))_{N_{n-2}},\delta_{B_{n-2}}^{1/2}\eta^{(n-2)})\ne0.
\end{align*}
This is condition \eqref{eq:assumption main proposition for only if characterization} for $n-2$ and implies, by Proposition~\ref{proposition:only if direction, for any principal series} which is assumed to hold for $n-2$, that $\eta^{(n-2)}$ satisfies $(\star)$ and then so does $\eta$. The proof of
the lemma is complete.
\begin{proof}[Proof of Claim~\ref{claim:t_x acts trivially}]
Let $x\in F^{*\mathe}$ and write $t_x=z_xd_x$ with
\begin{align*}
z_x=xI_n\in Z_n^{\mathe},\quad d_x=diag(x^{-1}I_{n-i},1,x^{-1}I_{i-2},1).
\end{align*}
Because $\rconj{w_i^{-1}}d_x=diag(x^{-1}I_{n-2},I_2)\in Z_{n-2}^{\mathe}$,
$\omega_{\theta_{n-2}}(\mathfrak{s}(\rconj{w_i^{-1}}d_x))=\gamma_{\psi}^{-1}(x^{-1})\gamma(x^{-1})$. Hence 
$d_x$ acts on $\rconj{w_i}(\Lambda(b)_{N_{n-2}})$
by
\begin{align*}
|x|^{-n+2}\gamma_{\psi}^{-1}(x^{-1})\gamma(x^{-1})\gamma_{\psi^{-1}}^{-1}(x^{-1})\gamma'(x^{-1})=
|x|^{-n+2}\gamma(x^{-1})\gamma'(x^{-1}).
\end{align*}
($\gamma_{\psi^{-1}}=\gamma_{\psi}^{-1}$.) The element $z_x$ acts on
$\omega_{\theta}\omega_{\theta'}\ind_{\rconj{w_i}T_{n-2}}^{T_n^{\circ}}(\delta_{Y_i}^{-1}\ \rconj{w_i}(\Lambda(b)_{N_{n-2}}))$
by $\omega_{\theta}\omega_{\theta'}(\mathfrak{s}(z_x))=\gamma(x)\gamma'(x)$. Since $\delta_{Y_i}^{-1}(d_x)=|x|^{i-2}$
and
\begin{align*}
\gamma(x)\gamma'(x)\gamma(x^{-1})\gamma'(x^{-1})=1
\end{align*}
($\gamma(x)\gamma(x^{-1})=(x,x^{-1})^{\lfloor n/2\rfloor}$),
$t_x$ acts on $\omega_{\theta}\omega_{\theta'}\ind_{\rconj{w_i}T_{n-2}}^{T_n^{\circ}}(\delta_{Y_i}^{-1}\ \rconj{w_i}(\Lambda(b)_{N_{n-2}}))$ via
$|x|^{i-n}=\delta_{B_n}^{1/2}(t_x)$.
\end{proof}
\end{proof}

\subsection*{Acknowledgments}
I wish to express my gratitude to Erez Lapid for suggesting this project to me. 
I thank Anthony Kable for helpful conversations.

\providecommand{\bysame}{\leavevmode\hbox to3em{\hrulefill}\thinspace}
\providecommand{\MR}[1]{}
\providecommand{\MRhref}[2]{%
  \href{http://www.ams.org/mathscinet-getitem?mr=#1}{#2}
}
\providecommand{\href}[2]{#2}

\end{document}